\documentclass[11pt,reqno]{amsart}

\usepackage{amsmath,amssymb,mathrsfs}
\usepackage{graphicx,cite,cases}
\usepackage{times}
\newtheorem{theo}{Theorem}[section]

\newtheorem{lemm}[theo]{Lemma}
\newtheorem{prop}[theo]{Proposition}
\newtheorem{rema}[theo]{Remark}

\numberwithin{equation}{section}

\begin{document}

\title[adaptive PML for elastic waves]{An adaptive finite element PML method for
the elastic wave scattering problem in periodic structures}

\author{Xue Jiang}
\address{School of Science, Beijing University of Posts and
Telecommunications, Beijing 100876, China.}
\email{jxue@lsec.cc.ac.cn}

\author{Peijun Li}
\address{Department of Mathematics, Purdue University, West Lafayette, IN 47907,
USA.}
\email{lipeijun@math.purdue.edu}

\author{Junliang Lv}
\address{School of Mathematics, Jilin University, Changchun 130012, China. }
\email{lvjl@jlu.edu.cn}

\author{Weiying Zheng}
\address{CMIS, LSEC, ICMSEC, Academy of Mathematics and System Sciences, Chinese
Academy of Sciences, Beijing, 100190, China.}
\email{zwy@lsec.cc.ac.cn}

\thanks{The research of XJ was supported in part by China NSF grant 11401040 and
by the Fundamental Research Funds for the Central Universities
24820152015RC17. The research of PL was supported in part by the NSF grant
DMS-1151308. The research of JL was partially supported by the China NSF grants
11126040 and 11301214. The research of WZ was supported in part by China NSF
grants 11171334 and 91430215, by the Funds for Creative Research Groups of China
(Grant No. 11021101), and by National 863 Project of China under the
grant 2012AA01A309.}



\begin{abstract}
An adaptive finite element method is presented for the elastic scattering of a
time-harmonic plane wave by a periodic surface. First, the unbounded physical
domain is truncated into a bounded computational domain by introducing
the perfectly matched layer (PML) technique. The well-posedness and exponential
convergence of the solution are established for the truncated PML problem by
developing an equivalent transparent boundary condition. Second, an a posteriori
error estimate is deduced for the discrete problem and is used to determine the
finite elements for refinements and to determine the PML parameters. Numerical
experiments are included to demonstrate the competitive behavior of the proposed
adaptive method.
\end{abstract}

\maketitle

\section{Introduction}

The scattering theory in periodic diffractive structures, which are known as
diffraction gratings, has many significant applications in optical industry
\cite{bcm-01, bdc-josa95}. The time-harmonic problems have been
studied extensively in diffraction gratings by many researchers for acoustic,
electromagnetic, and elastic waves \cite{a-jiea99, a-mmas99, b-sjna95, b-sjam97,
cf-tams91, df-jmaa92, eh-mmas10, eh-mmmas12, lwz-ip15, wbllw-sjna15}. The
underlying equations of these waves are the Helmholtz equation, the Maxwell
equations, and the Navier equation, respectively. This paper is concerned with
the numerical solution of the elastic wave scattering problem in such a periodic
structure. The problem has two fundamental challenges. The first one is to
truncate the unbounded physical domain into a bounded computational domain. The
second one is the singularity of the solution due to nonsmooth grating surfaces.
Hence, the goal of this work is two fold to overcome these two issues. First, we
adopt the perfectly matched layer (PML) technique to handle the domain
truncation. Second, we use an a posteriori error analysis and design a finite
element method with adaptive mesh refinements to deal with the singularity of
the solution.

The research on the PML technique has undergone a tremendous development since
B\'{e}renger proposed a PML for solving the time-dependent Maxwell equations
\cite{b-jcp94}. The basis idea of the PML technique is to surround the domain of
interest by a layer of finite thickness fictitious material which absorbs all
the waves coming from inside the computational domain. When the waves
reach the outer boundary of the PML region, their energies are so small that
the simple homogeneous Dirichlet boundary conditions can be imposed. Various
constructions of PML absorbing layers have been proposed and investigated
for the acoustic and electromagnetic wave scattering problems \cite{bw-sjna05,
bp-mc07, cm-sjsc98, ct-g01, cw-motl94, hsz-sjma03, ls-c98, ty-anm98}. The PML
technique is much less studied for the elastic wave scattering problems
\cite{hsb-jasa96}, especically for the rigorous convergence analysis. We refer
to \cite{bpt-mc10, cxz-mc} for recent study on convergence analysis of
the elastic obstacle scattering problem. Combined with the PML technique, an
effective adaptive finite element method was proposed in \cite{bcw-josa05,
cw-sjna03} to solve the two-dimensional diffraction grating problem where the
one-dimensional grating structure was considered. Due to the competitive
numerical performance, the method was quickly adopted to solve many other
scattering problems including the obstacle scattering problems \cite{cl-sjna05,
cc-mc08} and the three-dimensional diffraction grating problem \cite{blw-mc10}.
Based on the a posteriori error analysis, the adaptive finite element PML method
provides an effective numerical strategy which can be used to solve a variety of
wave propagation problems which are imposed in unbounded domains.

In this paper, we explore the possibility of applying such an adaptive finite
element PML method to solve the diffraction grating problem of elastic
waves. Specifically, we consider the incidence of a time-harmonic elastic plane
wave on a one-dimensional grating surface, which is assumed to be elastically
rigid. The open space, which is above the surface, is assumed to be filled with
a homogeneous and isotropic elastic medium. Using the quasi-periodicity of the
solution and the transparent boundary condition, we formulate the scattering
problem equivalently into a boundary value problem in a bounded domain. The
conservation of energy is proved for the model problem and is used to verify
our numerical results when the exact solutions are not available. Following the
complex coordinate stretching, we study the truncated PML problem which is an
approximation to the original scattering problem. We develop the transparent
boundary condition for the truncated PML problem and show that it has a unique
weak solution which converges exponentially to the solution of the original
scattering problem. Moreover, an a posteriori error estimate is deduced for the
discrete PML problem. It consists of the finite element error and the PML
modeling error. The estimate is used to design the adaptive finite element
algorithm to choose elements for refinements and to determine the PML
parameters. Numerical experiments show that the proposed method can effectively
overcome the aforementioned two challenges.

This paper presents a nontrivial application of the adaptive finite element
PML method for the grating problem from the Helmholtz (acoustic) and Maxwell
(electromagnetic) equations to the Navier (elastic) equation. The elastic wave
equation is complicated due to the coexistence of compressional and shear waves
that have different wavenumbers and propagate at different speeds. In view of
this physical feature, we introduce two scalar potential functions to split the
wave field into its compressional and shear parts via the Helmholtz
decomposition. As a consequence, the analysis is much more sophisticated than
that for the Helmholtz equation or the Maxwell equations. We believe that
this work not only enriches the range of applications for the PML technique but
also is a valued contribution to the family of numerical methods for solving
elastic wave scattering problems.

The paper is organized as follows. In section 2, we introduce the model problem
of the elastic wave scattering by a periodic surface and formulate it into a
boundary value problem by using a transparent boundary condition. The
conservation of the total energy is proved for the propagating wave modes. In
section 3, we introduce the PML formulation and prove the well-posedness and
convergence of the truncated PML problem. Section 4 is devoted to the finite
element approximation and the a posteriori error estimate. In section 5, we
discuss the numerical implementation of our adaptive algorithm and present some
numerical experiments to illustrate the performance of the proposed method. The
paper is concluded with some general remarks and directions for future research
in section 6.

\section{Problem formulation}

In this section, we introduce the model problem and present an exact transparent
boundary condition to reduce the problem into a boundary value problem in a
bounded domain. The energy distribution will be studied for the reflected
propagating waves of the scattering problem.

\subsection{Navier equation}

Consider the elastic scattering of a time-harmonic plane wave by a periodic
surface $S$ which is assumed to be Lipschitz continuous and elastically rigid.
In this work, we consider the two-dimensional problem by assuming that the
surface is invariant in the $z$ direction. The three-dimensional problem will
be studied as a separate work. Figure \ref{pg} shows the problem geometry in one
period. Let $\boldsymbol{x}=[x, y]^\top\in\mathbb{R}^2$. Denote
by $\Gamma=\{\boldsymbol{x}\in\mathbb{R}^2: 0<x<\Lambda,\, y=b\}$ the
artificial boundary above the scattering surface, where $\Lambda$ is the period
and $b$ is a constant. Let $\Omega$ be the bounded domain which is enclosed from
below and above by $S$ and $\Gamma$, respectively. Finally, denote by
$\Omega^{\rm e}=\{\boldsymbol{x}\in\mathbb{R}^2: 0<x<\Lambda,\, y>b\}$ the
exterior domain to $\Omega$.

The open space, which is above the grating surface, is assumed to be filled with
a homogeneous and isotropic elastic medium with a unit mass density.
The propagation of a time-harmonic elastic wave is governed by the Navier
equation
\begin{equation}\label{une}
 \mu\Delta\boldsymbol{u}+(\lambda+\mu)\nabla\nabla\cdot\boldsymbol{u}
+\omega^2\boldsymbol{u}=0\quad\text{in} ~ \Omega\cup\Omega^{\rm e},
\end{equation}
where $\omega>0$ is the angular frequency, $\mu$ and $\lambda$ are the Lam\'{e}
constants satisfying $\mu>0$ and $\lambda+\mu>0$,  and $\boldsymbol{u}=[u_1,
u_2]^\top$ is the displacement vector of the total field which satisfies
\begin{equation}\label{bc}
 \boldsymbol{u}=0\quad\text{on} ~ S.
\end{equation}

\begin{figure}
\centering
\includegraphics[width=0.32\textwidth]{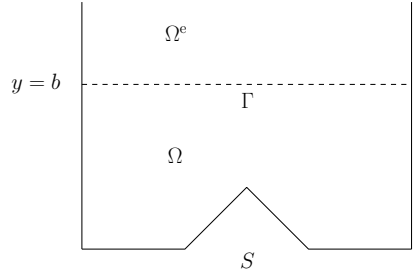}
\caption{Geometry of the scattering problem.}
\label{pg}
\end{figure}

Let the surface be illuminated from above by either a time-harmonic
compressional plane wave
\[
 \boldsymbol{u}_{\rm inc}(\boldsymbol x)=[\sin\theta,\,-\cos\theta]^\top
e^{{\rm i}\kappa_1(x\sin\theta - y\cos\theta)},
\]
or a time-harmonic shear plane wave
\[
 \boldsymbol{u}_{\rm inc}(\boldsymbol x)=[\cos\theta,\,\sin\theta]^\top e^{{\rm
i}\kappa_2 (x\sin\theta-y\cos\theta)},
\]
where $\theta\in(-\pi/2, \pi/2)$ is the incident angle and
\begin{equation}\label{wn}
 \kappa_1=\frac{\omega}{\sqrt{\lambda+2\mu}},\quad
\kappa_2=\frac{\omega}{\sqrt{\mu}}
\end{equation}
are the compressional and shear wavenumbers, respectively. It can be verified
that the incident wave also satisfies the Navier equation:
\begin{equation}\label{uine}
\mu\Delta\boldsymbol{u}_{\rm inc}+(\lambda
+\mu)\nabla\nabla\cdot\boldsymbol{u}_{\rm inc}
+\omega^2\boldsymbol{u}_{\rm inc}=0\quad\text{in} ~ \Omega\cup\Omega^{\rm e}.
\end{equation}

\begin{rema}
Our method works for either the compressional plane incident wave, or
the shear plane incident wave, or any linear combination of these two plane
incident waves. For clarity, we will take the compressional plane incident wave
as an example to present the results in our subsequent analysis.
\end{rema}

Motivated by uniqueness, we are interested in a quasi-periodic solution of
$\boldsymbol{u}$, i.e., $\boldsymbol{u}(x, y)e^{-{\rm i}\alpha x}$ is periodic
in $x$ with period $\Lambda$ where $\alpha=\kappa_1\sin\theta$. In addition, the
following radiation condition is imposed: the total displacement
$\boldsymbol{u}$ consists of bounded outgoing waves plus the incident wave
$\boldsymbol{u}_{\rm inc}$ in $\Omega^{\rm e}$.

We introduce some notation and Sobolev spaces. Let $\boldsymbol{u}=[u_1,
u_2]^\top$ and $u$ be a vector and scalar function, respectively. Define the
Jacobian matrix of $\boldsymbol{u}$ as
\[
 \nabla\boldsymbol{u}=\begin{bmatrix}
                       \partial_x u_1 & \partial_y u_1\\
                       \partial_x u_2 & \partial_y u_2
                      \end{bmatrix}
\]
and two curl operators
\[
 {\rm curl}\boldsymbol{u}=\partial_x u_2 - \partial_y u_1,\quad {\bf
curl} u=[\partial_y u, -\partial_x u]^\top.
\]
Define a quasi-periodic functional space
\[
 H^1_{S, \rm qp}(\Omega)=\{u\in H^1(\Omega): u(\Lambda, y)=u(0, y)e^{{\rm
i}\alpha\Lambda},\, u=0 ~\text{on}~S\},
\]
which is a subspace of $H^1(\Omega)$ with the norm $\|\cdot\|_{H^1(\Omega)}$.
For any quasi-periodic function $u$ defined on $\Gamma$, it admits the Fourier
series expansion
\[
 u(x)=\sum_{n\in\mathbb{Z}}u^{(n)}e^{{\rm i}\alpha_n x}, \quad
u^{(n)}=\frac{1}{\Lambda}\int_0^\Lambda u(x)e^{-{\rm i}\alpha_n x}{\rm
d}x,\quad\alpha_n=\alpha+n\left(\frac{2\pi}{\Lambda}\right).
\]
We define a trace functional space $H^s(\Gamma)$ with the norm
given by
\[
 \|u\|_{H^s(\Gamma)}=\Bigl(\Lambda \sum_{n\in\mathbb{Z}}(1+\alpha_n^2)^s
|u^{(n)}|^2\Bigr)^{1/2}.
\]
Let $H^1_{S, \rm qp}(\Omega)^2$ and $H^s(\Gamma)^2$ be the Cartesian product
spaces equipped with the corresponding 2-norms of $H^1_{S, \rm qp}(\Omega)$ and
$H^s(\Gamma)$, respectively. It is known that $H^{-s}(\Gamma)^2$ is the
dual space of $H^s(\Gamma)^2$ with respect to the $L^2(\Gamma)^2$ inner
product
\[
 \langle \boldsymbol{u},
\boldsymbol{v}\rangle_{\Gamma}=\int_{\Gamma}\boldsymbol{u}\cdot
\bar{\boldsymbol{v}}\,{\rm d}x,
\]
where the bar denotes the complex conjuate.

\subsection{Boundary value problem}

We wish to reduce the problem equivalently into a boundary value problem in
$\Omega$ by introducing an exact transparent boundary condition on $\Gamma$.

The total field $\boldsymbol{u}$ consists of the incident field
$\boldsymbol{u}_{\rm inc}$ and the diffracted field $\boldsymbol{v}$, i.e.,
\begin{equation}\label{tf}
 \boldsymbol{u}=\boldsymbol{u}_{\rm inc}+\boldsymbol{v}.
\end{equation}
Noting \eqref{tf} and subtracting \eqref{uine} from \eqref{une}, we obtain the
Navier equation for the diffracted field
$\boldsymbol{v}$:
\begin{equation}\label{vne}
\mu\Delta\boldsymbol{v}+(\lambda +\mu)\nabla\nabla\cdot\boldsymbol{v}
+\omega^2\boldsymbol{v}=0\quad\text{in} ~ \Omega^{\rm e}.
\end{equation}
For any solution $\boldsymbol{v}$ of \eqref{vne}, we introduce the Helmholtz
decomposition to split it into the compressional and shear parts:
\begin{equation}\label{hdv}
 \boldsymbol{v}=\nabla\phi_1 +{\bf curl}\phi_2,
\end{equation}
where $\phi_1$ and $\phi_2$ are scalar potential functions. Substituting
\eqref{hdv} into \eqref{vne} gives
\[
 \nabla\left((\lambda +2\mu)\Delta\phi_1 +\omega^2\phi_1 \right)+{\bf
curl}(\mu\Delta\phi_2 +\omega^2\phi_2)=0,
\]
which is fulfilled if $\phi_j$ satisfy the Helmholtz equation
\begin{equation}\label{he}
 \Delta\phi_j +\kappa_j^2\phi_j=0,
\end{equation}
where $\kappa_j$ is the wavenumber defined in \eqref{wn}.

Since $\boldsymbol{v}$ is a quasi-periodic function, we have from \eqref{hdv}
that $\phi_j$ is also a quasi-periodic function in the $x$ direction with
period $\Lambda$ and it has the Fourier series expansion
\begin{equation}\label{fse}
 \phi_j(x, y)=\sum_{n\in\mathbb{Z}}\phi_j^{(n)}(y)e^{{\rm i}\alpha_n x}.
\end{equation}
Plugging \eqref{fse} into \eqref{he} yields
\begin{equation}\label{ode}
 \frac{{\rm d}^2\phi_j^{(n)}(y)}{{\rm d}y^2}+\bigl(\beta_j^{(n)}\bigr)^2
\phi_j^{(n)}(y)=0, \quad y>b,
\end{equation}
where
\begin{equation}\label{beta}
\beta^{(n)}_j=
\begin{cases}
(\kappa_j^2 - \alpha_n^2)^{1/2},\quad &|\alpha_n|<\kappa_j,\\[2pt]
{\rm i}(\alpha_n^2 - \kappa_j^2)^{1/2},\quad &|\alpha_n|>\kappa_j.
\end{cases}
\end{equation}
Note that $\beta_1^{(0)}=\beta=\kappa_1\cos\theta$. We assume that $\kappa_j\neq
|\alpha_n|$ for all $n\in\mathbb{Z}$ to exclude possible resonance. Noting
\eqref{beta} and using the bounded outgoing radiation condition, we obtain the
solution of \eqref{ode}:
\[
 \phi_j^{(n)}(y)=\phi_j^{(n)}(b) e^{{\rm i}\beta_j^{(n)}(y-b)},
\]
which gives Rayleigh's expansion for $\phi_j$:
\begin{equation}\label{pfre}
 \phi_j(x, y)=\sum_{n\in\mathbb{Z}}\phi_j^{(n)}(b) e^{{\rm i}\bigl(\alpha_n
x+\beta_j^{(n)}(y-b)\bigr)},\quad y>b.
\end{equation}
Combining \eqref{pfre} and the Helmholtz decomposition \eqref{hdv} yields
\begin{equation}\label{vre}
 \boldsymbol{v}(x, y)={\rm i}\sum_{n\in\mathbb{Z}}
 \begin{bmatrix}
  \alpha_n\\[5pt]
  \beta_1^{(n)}
 \end{bmatrix}
\phi_1^{(n)}(b) e^{{\rm i}\bigl(\alpha_n x
+\beta_1^{(n)}(y-b)\bigr)} +\begin{bmatrix}
  \beta_2^{(n)}\\[5pt]
  -\alpha_n
 \end{bmatrix}
\phi_2^{(n)}(b) e^{{\rm i}\bigl(\alpha_n x +\beta_2^{(n)}(y-b)\bigr)}.
\end{equation}

On the other hand, as a quasi-periodic function, the diffracted field
$\boldsymbol{v}$ also has the Fourier series expansion
\begin{equation}\label{vfe}
 \boldsymbol{v}(x, b)=\sum_{n\in\mathbb{Z}}\boldsymbol{v}^{(n)}(b) e^{{\rm
i}\alpha_n x}.
\end{equation}
From \eqref{vfe} and \eqref{vre}, we obtain a linear
system of algebraic equations for $\phi_j^{(n)}(b)$:
\[
 \begin{bmatrix}
  {\rm i}\alpha_n & {\rm i}\beta_2^{(n)}\\[5pt]
  {\rm i}\beta_1^{(n)} & -{\rm i}\alpha_n
 \end{bmatrix}
 \begin{bmatrix}
  \phi_1^{(n)}(b)\\[5pt]
  \phi_2^{(n)}(b)
 \end{bmatrix}=
 \begin{bmatrix}
  v_1^{(n)}(b)\\[5pt]
  v_2^{(n)}(b)
 \end{bmatrix}.
 \]
Solving the above equations via Cramer's rule gives
\begin{subequations}\label{pffc}
\begin{align}
\phi_1^{(n)}(b)&=-\frac{\rm i}{\chi^{(n)}}\bigl(\alpha_n
v_1^{(n)}(b)+\beta_2^{(n)} v_2^{(n)}(b)\bigr),\\
\phi_2^{(n)}(b)&=-\frac{\rm i}{\chi^{(n)}}\bigl(\beta_1^{(n)}
v_1^{(n)}(b)-\alpha_n v_2^{(n)}(b)\bigr),
\end{align}
\end{subequations}
where
\begin{equation}\label{chi}
 \chi^{(n)}=\alpha_n^2 + \beta_1^{(n)}\beta_2^{(n)}.
\end{equation}
Plugging \eqref{pffc} into \eqref{vre}, we obtain Rayleigh's expansion for the
diffracted field $\boldsymbol{v}$ in $\Omega^{\rm e}$:
\begin{align}\label{vr}
 \boldsymbol{v}(x, y)=\sum_{n\in\mathbb{Z}} \frac{1}{\chi^{(n)}}
&\begin{bmatrix}
\alpha_n^2 &  \alpha_n\beta_2^{(n)}\\[5pt]
\alpha_n \beta_1^{(n)} & \beta_1^{(n)}\beta_2^{(n)}
\end{bmatrix}
\boldsymbol{v}^{(n)}(b)e^{{\rm i}\bigl(\alpha_n
x+\beta_1^{(n)}(y-b)\bigr)}\notag\\
&+\frac{1}{\chi^{(n)}}
\begin{bmatrix}
\beta_1^{(n)}\beta_2^{(n)} & - \alpha_n\beta_2^{(n)}\\[5pt]
-\alpha_n \beta_1^{(n)} & \alpha_n^2
\end{bmatrix}
\boldsymbol{v}^{(n)}(b)e^{{\rm i}\bigl(\alpha_n
x+\beta_2^{(n)}(y-b)\bigr)}.
\end{align}

Given a vector field $\boldsymbol{v}=[v_1, v_2]^\top$, we define a differential
operator on $\Gamma$:
\begin{equation}\label{do}
 \mathscr{D}\boldsymbol{v}=\mu\partial_y\boldsymbol{v}+(\lambda +\mu)[0,
1]^\top\nabla\cdot\boldsymbol{v}=[\mu\partial_y v_1, (\lambda
+\mu)\partial_x v_1+(\lambda +2\mu)\partial_y v_2]^\top.
\end{equation}
By \eqref{do}, and \eqref{vr}, we deduce the transparent
boundary condition
\[
\mathscr{D}\boldsymbol{v}=\mathscr{T}\boldsymbol{v}:=\sum_{n\in\mathbb{Z}}M^{(n)
} \boldsymbol{v}^{(n)}(b)e^{{\rm i}\alpha_n x}\quad\text{on} ~ \Gamma,
\]
where the matrix
\[
 M^{(n)}=\frac{\rm i}{\chi^{(n)}}
 \begin{bmatrix}
\omega^2 \beta_1^{(n)} & \mu\alpha_n \chi^{(n)} -\omega^2\alpha_n\\[5pt]
\omega^2\alpha_n-\mu\alpha_n \chi^{(n)}& \omega^2\beta_2^{(n)}
 \end{bmatrix}.
\]
Equivalently, we have the transparent boundary condition for the total
field $\boldsymbol{u}$:
\[
 \mathscr{D}\boldsymbol{u}=\mathscr{T}\boldsymbol{u}+\boldsymbol{f}\quad\text{on
} ~ \Gamma,
\]
where $\boldsymbol{f}=\mathscr{D}\boldsymbol{u}_{\rm
inc}-\mathscr{T}\boldsymbol{u}_{\rm inc}$.

The scattering problem can be reduced to the following boundary value problem:
\begin{equation}\label{bvp}
 \begin{cases}
  \mu\Delta\boldsymbol{u}+(\lambda+\mu)\nabla
\nabla\cdot\boldsymbol{u}+\omega^2\boldsymbol{u}=0\quad&\text{in}~ \Omega,\\
\boldsymbol{u}=0\quad&\text{on} ~ S,\\
\mathscr{D}\boldsymbol{u}=\mathscr{T}\boldsymbol{u}+\boldsymbol{f}\quad&\text{on
} ~ \Gamma.
 \end{cases}
\end{equation}
The weak formulation of \eqref{bvp} reads as follows: Find $\boldsymbol{u}\in
H^1_{S, \rm qp}(\Omega)^2$ such that
\begin{equation}\label{wp}
 a(\boldsymbol{u}, \boldsymbol{v})=\langle\boldsymbol{f},
\boldsymbol{v}\rangle_{\Gamma},\quad\forall\,\boldsymbol{v}\in H^1_{S, \rm
qp}(\Omega)^2,
\end{equation}
where the sesquilinear form $a: H^1_{S, \rm qp}(\Omega)^2\times H^1_{S, \rm
qp}(\Omega)^2\to\mathbb{C}$ is defined by
\begin{align}\label{sf}
 a(\boldsymbol{u}, \boldsymbol{v})=\mu\int_\Omega
\nabla\boldsymbol{u}:\nabla\bar{\boldsymbol
v}\,{\rm d}\boldsymbol{x}+(\lambda+\mu)\int_\Omega (\nabla\cdot\boldsymbol{u}
)(\nabla\cdot\bar{\boldsymbol v})\,{\rm d}\boldsymbol{x}\notag\\
-\omega^2\int_\Omega \boldsymbol{u}\cdot\bar{\boldsymbol v}\, {\rm
d}\boldsymbol{x}- \langle T\boldsymbol{u}, \boldsymbol{v}\rangle_{\Gamma}.
\end{align}
Here $A:B={\rm tr}(A B^\top)$ is the Frobenius inner product of square matrices
$A$ and $B$.

The existence of a unique weak solution $\boldsymbol{u}$ of \eqref{wp} is
discussed in \cite{eh-mmas10}. In this paper, we assume that the variational
problem \eqref{wp} admits a unique solution. It follows from the general theory
in \cite{ba-73} that there exists a constant $\gamma_1>0$ such that the
following inf-sup condition holds
\begin{equation}\label{ifc}
\sup_{0\neq \boldsymbol{v}\in H^1_{S, \rm qp}(\Omega)^2}\frac{|a(\boldsymbol{u},
\boldsymbol{v})|}{\|\boldsymbol{v}\|_{H^1(\Omega)^2}}\geq \gamma_1
\|\boldsymbol{u}\|_{H^1(\Omega)^2},\quad\forall\,\boldsymbol{u}\in H^1_{S, \rm
qp}(\Omega)^2.
\end{equation}

\subsection{Energy distribution}

We study the energy distribution for the propagating reflected wave modes of
the displacement. The result will be used to verify the accuracy of our
numerical method when the analytic solution is not available.

Denote by $\boldsymbol{\nu}=(\nu_1, \nu_2)^\top$ and $\boldsymbol{\tau}=(\tau_1,
\tau_2)^\top$ the unit normal and tangential vectors on $S$, where
$\tau_1=\nu_2$ and $\tau_2=-\nu_1$. Let
$\Delta_j^{(n)}=|\kappa_j^2-\alpha_n^2|^{1/2}$ and $U_j=\{n:
|\alpha_n|<\kappa_j\}$. We point out that $U_1$ and $U_2$ are the
collections of all the propagating modes for the compressional and shear waves,
respectively. It is clear to note that $\beta_j^{(n)}=\Delta_j^{(n)}$ for $n\in
U_j$ and $\beta_j^{(n)}={\rm i}\Delta_j^{(n)}$ for $n\notin U_j$.

Consider the Helmholtz decomposition for the total field:
\begin{equation}\label{hdt}
 \boldsymbol{u}=\nabla\varphi_1+{\bf curl}\varphi_2.
\end{equation}
Substituting \eqref{hdt} into \eqref{une}, we may verify that $\varphi_j$ also
satisfies the Helmholtz equation
\[
 \Delta \varphi_j+\kappa_j^2\varphi_j=0\quad\text{in} ~ \Omega\cup\Omega^{\rm
e}.
\]
Using the boundary condition \eqref{bc}, we have
\[
 \partial_{\boldsymbol\nu}\varphi_1-\partial_{\boldsymbol\tau}
\varphi_2=0\quad\text{and}\quad \partial_{\boldsymbol\nu}\varphi_2+\partial_{
\boldsymbol\tau}\varphi_1=0 \quad\text{on} ~ S.
\]

Correspondingly, we introduce the Helmholtz decomposition for the incident
field:
\[
 \boldsymbol{u}_{\rm inc}=\nabla\psi_1 + {\bf curl}\psi_2,
\]
which gives explicitly that
\[
 \psi_1=-\frac{1}{\kappa_1^2}\nabla\cdot\boldsymbol{u}_{\rm inc}=-\frac{\rm
i}{\kappa_1}e^{{\rm i}(\alpha x-\beta y)},\quad\psi_2=\frac{1}{\kappa_2^2}{\rm
curl}\boldsymbol{u}_{\rm inc}=0.
\]
Hence we have
\[
 \varphi_1=\phi_1+\psi_1, \quad \varphi_2=\phi_2.
\]
Using the Rayleigh expansions \eqref{pfre}, we get
\begin{align}
\label{vphi1} \varphi_1(x, y)&=r_0 e^{{\rm i}(\alpha x-\beta
y)} +\sum_{n\in\mathbb{Z}}r_1^{(n)} e^{{\rm i}\bigl(\alpha_n
x+\beta_1^{(n)}y\bigr)},\\
\label{vphi2}\varphi_2(x, y)&=\sum_{n\in\mathbb{Z}}r_2^{(n)} e^{{\rm
i}\bigl(\alpha_n
x+\beta_2^{(n)}y\bigr)},
\end{align}
where
\[
r_0=-\frac{\rm i}{\kappa_1}, \quad r_1^{(n)}=\phi_1^{(n)}(b)e^{-{\rm
i}\beta_1^{(n)} b},\quad r_2^{(n)}=\phi_2^{(n)}(b)e^{-{\rm i}\beta_2^{(n)}b}.
\]

The grating efficiency is defined by
\begin{equation}\label{ge}
 e_1^{(n)}=\frac{\beta_1^{(n)}|r_1^{(n)}|^2}{\beta|r_0|^2},\quad
e_2^{(n)}=\frac{\beta_2^{(n)}|r_2^{(n)}|^2}{\beta|r_0|^2},
\end{equation}
where $e_1^{(n)}$ and $e_2^{(n)}$ are the efficiency of the $n$-th
order reflected modes for the compressional wave and the shear wave,
respectively. We have the following conservation of energy.

\begin{theo}\label{ce}
The total energy is conserved, i.e.,
\[
\sum_{n\in U_1}e_1^{(n)}+\sum_{n\in U_2}e_2^{(n)}=1.
\]
\end{theo}

\begin{proof}
Consider the following coupled problem:
\begin{equation}\label{ce-s1}
\begin{cases}
 \Delta\varphi_j + \kappa_j^2\varphi_j=0 &\quad\text{in} ~ \Omega,\\
 \partial_{\boldsymbol\nu}\varphi_1-\partial_{\boldsymbol\tau}
\varphi_2=0&\quad\text{on} ~ S,\\
\partial_{\boldsymbol\nu}\varphi_2+\partial_{\boldsymbol\tau}\varphi_1=0
&\quad\text{on} ~ S.
 \end{cases}
\end{equation}
It is clear to note that $\bar{\varphi}_j$ also satisfies the problem
\eqref{ce-s1} since the wavenumber $\kappa_j$ is real. Using Green's
theorem and quasi-periodicity of the solution, we have
\begin{align}\label{ce-s2}
0=&\int_\Omega(\bar{\varphi}_1\Delta\varphi_1-\varphi_1\Delta\bar{\varphi}_1)\,{
\rm
d}\boldsymbol{x}+(\bar{\varphi}_2\Delta\varphi_2-\varphi_2\Delta\bar{\varphi}
_2)\,{\rm d}\boldsymbol{x}\notag\\
=&\int_S(\bar{\varphi}_1\partial_{\boldsymbol\nu}\varphi_1-\varphi_1\partial_{
\boldsymbol\nu}\bar{\varphi}_1)\,{\rm d}s+\int_S
(\bar{\varphi}_2\partial_{\boldsymbol\nu}\varphi_2-\varphi_2\partial_{
\boldsymbol\nu}\bar{\varphi}_2)\,{\rm d}s\notag\\
& +\int_\Gamma (\bar{\varphi}_1\partial_y
\varphi_1-\varphi_1\partial_y\bar{\varphi}_1)\,{\rm d}x+\int_\Gamma
(\bar{\varphi}_2\partial_y\varphi_2-\varphi_2\partial_y\bar{\varphi}_2)\,{\rm
d}x.
\end{align}
It follows from integration by parts and the boundary conditions on $S$ in
\eqref{ce-s1} that
\begin{align*}
& \int_S \bar{\varphi}_1\partial_{\boldsymbol\nu}\varphi_1\,{\rm d}s=\int_S
\bar{\varphi}_1\partial_{\boldsymbol\tau}\varphi_2\,{\rm d}s=-\int_S
\varphi_2\partial_{\boldsymbol\tau}\bar{\varphi}_1\,{\rm d}s=\int_S
\varphi_2\partial_{\boldsymbol\nu}\bar{\varphi}_2\,{\rm d}s,\\
& \int_S \bar{\varphi}_2\partial_{\boldsymbol\nu}\varphi_2\,{\rm d}s=-\int_S
\bar{\varphi}_2\partial_{\boldsymbol\tau}\varphi_1\,{\rm d}s=\int_S
\varphi_1\partial_{\boldsymbol\tau}\bar{\varphi}_2\,{\rm d}s=\int_S
\varphi_1\partial_{\boldsymbol\nu}\bar{\varphi}_1\,{\rm d}s,
\end{align*}
which yields after taking the imaginary part of \eqref{ce-s2} that
\begin{equation}\label{ce-s3}
 {\rm Im}\int_\Gamma (\bar{\varphi}_1\partial_y\varphi_1 +
\bar{\varphi}_2\partial_y\varphi_2)\,{\rm d}x=0.
\end{equation}
It follows from \eqref{vphi1} and \eqref{vphi2} that we have
\begin{align*}
 \varphi_1(x, b)&=r_0 e^{{\rm i}(\alpha x-\beta
b)} +\sum_{n\in U_1}r_1^{(n)} e^{\bigl({\rm i}\alpha_n
x+{\rm i}\Delta_1^{(n)}b\bigr)} +\sum_{n\notin U_1}r_1^{(n)} e^{\bigl({\rm
i}\alpha_n x-\Delta_1^{(n)}b\bigr)},\\
\varphi_2(x, b)&=\sum_{n\in U_2}r_2^{(n)} e^{\bigl({\rm i}\alpha_n
x+{\rm i}\Delta_2^{(n)}b\bigr)}+\sum_{n\notin U_2}r_2^{(n)} e^{\bigl({\rm
i}\alpha_n x-\Delta_2^{(n)}b\bigr)},
\end{align*}
and
\begin{align*}
 \partial_y\varphi_1(x, b)&=-{\rm i}\beta r_0 e^{{\rm i}(\alpha x-\beta
b)}  +\sum_{n\in U_1} {\rm i}\Delta_1^{(n)} r_1^{(n)} e^{\bigl({\rm i}\alpha_n
x+{\rm i}\Delta_1^{(n)}b\bigr)} -\sum_{n\notin U_1} \Delta_1^{(n)} r_1^{(n)}
e^{\bigl({\rm i}\alpha_n x-\Delta_1^{(n)}b\bigr)},\\
\partial_y\varphi_2(x, b)&=\sum_{n\in U_2}{\rm i}\Delta_2^{(n)} r_2^{(n)}
e^{\bigl({\rm i}\alpha_n x+{\rm i}\Delta_2^{(n)}b\bigr)}-\sum_{n\notin
U_2}\Delta_2^{(n)} r_2^{(n)} e^{\bigl({\rm i}\alpha_n x-\Delta_2^{(n)}b\bigr)}.
\end{align*}
Substituting the above four functions into \eqref{ce-s3} and using the
orthogonality of Fourier series, we get
\[
 \sum_{n\in U_1}\Delta_1^{(n)}|r_1^{(n)}|^2 +\sum_{n\in
U_2}\Delta_2^{(n)}|r_2^{(n)}|^2=\beta |r_0|^2,
\]
which completes the proof.
\end{proof}

In practice, the grating efficiencies \eqref{ge} can be computed from
\eqref{pffc} once the scattering problem is solved and the diffracted field
$\boldsymbol{v}$ is available on $\Gamma$.

\section{The PML problem}

In this section, we shall introduce the PML formulation for the scattering
problem and establish the well-posedness of the PML problem. An error estimate
will be shown for the solutions between the original scattering problem and the
PML problem.

\subsection{PML formulation}

Now we turn to the introduction of an absorbing PML layer. As is shown in
Figure \ref{pg_pml}, the domain $\Omega$ is covered by a chunck of PML layer
of thickness $\delta$ in $\Omega^{\rm e}$. Let $\rho(\tau)=\rho_1(\tau)+{\rm
i}\rho_2(\tau)$ be the PML function which is continuous and satisfies
\[
 \rho_1=1, \quad \rho_2=0 \quad\text{for}~ \tau<b\quad\text{and}\quad
\rho_1\geq 1,\quad \rho_2>0\quad\text{otherwise}.
\]
We introduce the PML by complex coordinate stretching:
\begin{equation}\label{cs}
 \hat{y}=\int_0^y \rho(\tau) {\rm d}\tau.
\end{equation}

\begin{figure}
\centering
\includegraphics[width=0.36\textwidth]{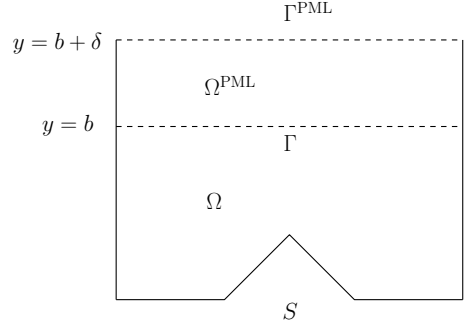}
\caption{Geometry of the PML problem.}
\label{pg_pml}
\end{figure}

Let $\hat{\boldsymbol x}=(x, \hat{y})$. Introduce the new field
\begin{equation}\label{nf}
 \hat{\boldsymbol u}(\boldsymbol{x})=\begin{cases}
\boldsymbol{u}_{\rm inc}(\boldsymbol{x})+(\boldsymbol{u}(\hat{\boldsymbol
x})-\boldsymbol{u}_{\rm inc}(\hat{\boldsymbol x})),\quad&
\boldsymbol{x}\in\Omega^{\rm e},\\
\boldsymbol{u}(\hat{\boldsymbol x}) ,\quad
&\boldsymbol{x}\in\Omega.
 \end{cases}
\end{equation}
It is clear to note that $\hat{\boldsymbol u}({\boldsymbol
x})=\boldsymbol{u}(\boldsymbol{x})$ in $\Omega$ since $\hat{\boldsymbol
x}=\boldsymbol{x}$ in $\Omega$. It can be verified from \eqref{une} and
\eqref{cs} that $\hat{\boldsymbol u}$ satisfies
\[
 \mathscr{L}(\hat{\boldsymbol u}-\boldsymbol{u}_{\rm inc})=0\quad\text{in}
~\Omega\cup \Omega^{\rm e}.
\]
Here the PML differential operator
\[
 \mathscr{L}\boldsymbol{u}:=\begin{bmatrix}
     (\lambda+2\mu)\partial_x (\rho(y)\partial_x u_1)+\mu\partial_y
(\rho^{-1}(y)\partial_y u_1)+(\lambda+\mu)\partial^2_{xy} u_2+\omega^2
\rho(y)u_1\\[5pt]
     \mu\partial_x(\rho(y)\partial_x u_2)+(\lambda+2\mu)\partial_y(\rho^{-1}(y)
\partial_y u_2)+(\lambda+\mu)\partial^2_{xy} u_1+\omega^2 \rho(y)u_2
    \end{bmatrix}.
\]

Define the PML region
\[
 \Omega^{\rm PML}=\{\boldsymbol{x}\in\mathbb{R}^2: 0<x<\Lambda,\,
b<y<b+\delta\}.
\]
Clearly, we have from \eqref{nf} and \eqref{vr} that the outgoing wave
$\hat{\boldsymbol u}(\boldsymbol{x})-\boldsymbol{u}_{\rm
inc}(\boldsymbol{x})$ in $\Omega^{\rm e}$ decays exponentially as
$y\to \infty$. Therefore, the homogeneous Dirichlet boundary condition can be
imposed on
\[
 \Gamma^{\rm PML}=\{\boldsymbol{x}\in\mathbb{R}^2:
0<x<\Lambda,\,y=b+\delta\}
\]
to truncate the PML problem. Define the computational domain for the PML
problem $D=\Omega\cup\Omega^{\rm PML}$. We arrive at the following truncated PML
problem: Find a quasi-periodic solution $\hat{\boldsymbol u}$ such that
\begin{equation}\label{pmlp}
 \begin{cases}
 \mathscr{L}\hat{\boldsymbol u}=\boldsymbol{g}&\quad\text{in} ~ D,\\
  \hat{\boldsymbol u}=\boldsymbol{u}_{\rm inc}&\quad\text{on} ~
\Gamma^{\rm PML},\\
\hat{\boldsymbol u}=0&\quad\text{on} ~ S,
 \end{cases}
\end{equation}
where
\[
 \boldsymbol{g}=\begin{cases}
                 \mathscr{L}\boldsymbol{u}_{\rm inc}&\quad\text{in} ~
\Omega^{\rm PML},\\
                 0&\quad\text{in} ~ \Omega.
                \end{cases}
\]

Define $H^1_{0, \rm qp}(D)=\{u\in H^1_{\rm qp}(D):
u=0~\text{on}~S\cup\Gamma^{\rm
PML}\}$. The weak formulation of the PML problem \eqref{pmlp} reads as
follows: Find $\hat{\boldsymbol u}\in H^1_{S, \rm qp}(D)^2$ such that
$\hat{\boldsymbol u}=\boldsymbol{u}_{\rm inc}$ on $\Gamma^{\rm PML}$ and
\begin{equation}\label{twp}
 b_D(\hat{\boldsymbol u}, \boldsymbol{v})=-\int_D
\boldsymbol{g}\cdot\bar{\boldsymbol v}{\rm d}\boldsymbol{x},\quad\forall\,
\boldsymbol{v}\in H^1_{0, \rm qp}(D)^2.
\end{equation}
Here for any domain $G\subset\mathbb{R}^2$, the sesquilinear form $b_G: H^1_{\rm qp}(G)^2\times
H^1_{\rm qp}(G)^2\to\mathbb{C}$ is defined by
\begin{align*}
 b_G(\boldsymbol{u}, \boldsymbol{v})=\int_G (\lambda+2\mu)
(\rho\partial_x u_1\partial_x\bar{v}_1+\rho^{-1}\partial_y
u_2\partial_y\bar{v}_2)+\mu(\rho^{-1}\partial_y
u_1\partial_y\bar{v}_1+\rho\partial_x u_2\partial_x\bar{v}_2)\\
+(\lambda+\mu)(\partial_x u_2\partial_y\bar{v}_1+\partial_x
u_1\partial_y\bar{v}_2)-\omega^2\rho(u_1\bar{v}_1+u_2\bar{v}_2)\,{\rm
d}\boldsymbol{x}.
\end{align*}

We will reformulate the variational problem \eqref{twp} in the domain
$D$ into an equivalent variational formulation in the domain $\Omega$, and
discuss the existence and uniqueness of the weak solution to the equivalent weak
formulation. To do so, we need to introduce the transparent boundary condition
for the truncated PML problem.

\subsection{Transparent boundary condition of the PML problem}

Let $\hat{\boldsymbol v}(\boldsymbol{x})=\boldsymbol{v}(\hat{\boldsymbol
x})=\boldsymbol{u}(\hat{\boldsymbol x})-\boldsymbol{u}_{\rm
inc}(\hat{\boldsymbol x})$. It is clear to note that $\hat{\boldsymbol v}$
satisfies the Navier equation in the complex coordinate
\begin{equation}\label{cvne}
\mu\Delta_{\hat{\boldsymbol x}}\hat{\boldsymbol v}+(\lambda
+\mu)\nabla_{\hat{\boldsymbol x}}\nabla_{\hat{\boldsymbol
x}}\cdot\hat{\boldsymbol v} +\omega^2\hat{\boldsymbol v}=0\quad\text{in} ~
\Omega^{\rm PML},
\end{equation}
where $\nabla_{\hat{\boldsymbol x}}=[\partial_x, \partial_{\hat y}]^\top$ with
$\partial_{\hat y}=\rho^{-1}(y)\partial_y$.

We introduce the Helmholtz decomposition to the solution of \eqref{cvne}:
\begin{equation}\label{chdv}
 \hat{\boldsymbol v}=\nabla_{\hat{\boldsymbol x}}\hat{\phi}_1+{\bf
curl}_{\hat{\boldsymbol x}}\hat{\phi}_2,
\end{equation}
where ${\bf curl}_{\hat{\boldsymbol x}}=[\partial_{\hat y}, -\partial_x]^\top$
and $\hat{\phi}_j(\boldsymbol{x})=\phi_j(\hat{\boldsymbol{x}})$ satisfies the
Helmholtz equation
\begin{equation}\label{che}
 \Delta_{\hat{\boldsymbol{x}}}\hat{\phi}_j +\kappa_j^2\hat{\phi}_j=0.
\end{equation}
Due to the quasi-periodicity of the solution, we have the Fourier series
expansion
\begin{equation}\label{cfse}
 \hat{\phi}_j(x, y)=\sum_{n\in\mathbb{Z}}\hat{\phi}_j^{(n)}(y)e^{{\rm i}\alpha_n
x}.
\end{equation}
Substituting \eqref{cfse} into \eqref{che} yields
\begin{equation}\label{code}
 \rho^{-1}\frac{\rm d}{{\rm d}y}\Bigl(\rho^{-1}\frac{{\rm
d}}{{\rm d}y}\hat{\phi}_j^{(n)}(y)\Bigr)+(\beta_j^{(n)})^2
\hat{\phi}_j^{(n)}(y)=0.
\end{equation}
The general solutions of \eqref{code} is
\begin{align*}
 \hat{\phi}_j^{(n)}(y)=A_j^{(n)} e^{{\rm
i}\beta_j^{(n)}\int_{b}^y\rho(\tau){\rm d}\tau} + B_j^{(n)}e^{-{\rm
i}\beta_j^{(n)}\int_{b}^y \rho(\tau){\rm d}\tau}.
\end{align*}
Denote by
\begin{equation}\label{zeta}
\zeta=\int_b^{b+\delta}\rho(\tau){\rm d}\tau.
\end{equation}
The coefficients $A_j^{(n)}$ and $B_j^{(n)}$ can be uniquely determined by
solving the following linear equations
\begin{equation}\label{lev}
\begin{bmatrix}
 \alpha_n & \alpha_n & \beta_2^{(n)} & -\beta_2^{(n)}\\[5pt]
 \beta_1^{(n)} &-\beta_1^{(n)} & -\alpha_n & -\alpha_n\\[5pt]
 \alpha_ne^{{\rm i}\beta_1^{(n)}\zeta} & \alpha_ne^{-{\rm
i}\beta_1^{(n)}\zeta}&
 \beta_2^{(n)}e^{{\rm i}\beta_2^{(n)}\zeta} & -\beta_2^{(n)}e^{-{\rm
i}\beta_2^{(n)}\zeta}\\[5pt]
 \beta_1^{(n)}e^{{\rm i}\beta_1^{(n)}\zeta} & -\beta_1^{(n)}e^{-{\rm
i}\beta_1^{(n)}\zeta}&-\alpha_ne^{{\rm i}\beta_2^{(n)}\zeta} &
-\alpha_ne^{-{\rm i}\beta_2^{(n)}\zeta}
\end{bmatrix}
\begin{bmatrix}
 A_1^{(n)}\\[5pt]
 B_1^{(n)}\\[5pt]
 A_2^{(n)}\\[5pt]
 B_2^{(n)}
\end{bmatrix}
=\begin{bmatrix}
  -{\rm i}\hat{v}_1^{(n)}(b)\\[5pt]
  -{\rm i}\hat{v}_2^{(n)}(b)\\[5pt]
  0\\[5pt]
  0
 \end{bmatrix},
\end{equation}
where we have used the Helmholtz decomposition \eqref{chdv} and the homogeneous
Dirichlet boundary condition
\[
\hat{\boldsymbol v}(x,b+\delta)=0\quad\text{on} ~ \Gamma^{\rm PML}
\]
due to the PML absorbing layer. Solving the linear equations \eqref{lev}, we
obtain
\begin{align*}
  A_1^{(n)}=&\frac{{\rm i}}{2\chi^{(n)}\hat{\chi}^{(n)}}
  \Big\{-\chi^{(n)}(\varepsilon_1^{(n)}+2)
  (\alpha_n\hat{v}_1^{(n)}(b)+\beta_2^{(n)}\hat{v}_2^{(n)}(b))\\
  &+2\beta_2^{(n)}(\varepsilon_1^{(n)}
  +2\delta_1^{(n)})(1+\delta_2^{(n)}
  -\eta^{(n)})(\alpha_n\beta_1^{(n)}\hat{v}_1^{(n)}(b)
  +\alpha_n^2\hat{v}_2^{(n)}(b))\Big\},\\
  B_1^{(n)}=&\frac{{\rm i}}{2\chi^{(n)}\hat{\chi}^{(n)}}
  \Big\{\chi^{(n)}\varepsilon_1^{(n)}
  (\alpha_n\hat{v}_1^{(n)}(b)-\beta_2^{(n)}\hat{v}_2^{(n)}(b))\\
  &+2(\varepsilon_1^{(n)}\delta_2^{(n)}
  +2(\delta_1^{(n)}+\delta_1^{(n)}\delta_2^{(n)})
  (\alpha_n\beta_1^{(n)}\beta_2^{(n)}\hat{v}_1^{(n)}(b)
  -\alpha_n^2\beta_2^{(n)}\hat{v}_2^{(n)}(b))\Big\},\\
  A_2^{(n)}=&\frac{{\rm i}}{2\chi^{(n)}\hat{\chi}^{(n)}}
  \Big\{\chi^{(n)}[\varepsilon_1^{(n)}\eta^{(n)}
-2(\varepsilon_1^{(n)}+1)(1+\delta_2^{(n)})]
  (\beta_1^{(n)}\hat{v}_1^{(n)}(b)-\alpha_n\hat{v}_2^{(n)}(b))\\
&+2\varepsilon_1^{(n)}(1+\delta_2^{(n)}-\eta^{(n)})((\beta_1^{(n)})^2\beta_2^{
(n)}\hat{v}_1^{(n)}(b)-\alpha_n^3\hat{v}_2^{(n)}(b))\Big\},\\
  B_2^{(n)}=&\frac{{\rm i}}{2\chi^{(n)}\hat{\chi}^{(n)}}
  \Big\{\chi^{(n)}[2\delta_2^{(n)}(\varepsilon_1^{(n)}+1)
  -\varepsilon_1^{(n)}\eta^{(n)}]
  (\beta_1^{(n)}\hat{v}_1^{(n)}(b)+\alpha_n\hat{v}_2^{(n)}(b))\\
  &-2\delta_2^{(n)}(\varepsilon_1^{(n)}+2)
  ((\beta_1^{(n)})^2\beta_2^{(n)}\hat{v}_1^{(n)}(b)
  +\alpha_n^3\hat{v}_2^{(n)}(b))\Big\},
\end{align*}
where
\begin{equation}\label{vepsn}
\begin{cases}
\varepsilon_j^{(n)}&=\coth(-{\rm i}\beta_j^{(n)}\zeta)-1,\\
 \delta_j^{(n)}&=(e^{{\rm i}\beta_2^{(n)}\zeta}-e^{{\rm
i}\beta_1^{(n)}\zeta})/
  (e^{-{\rm i}\beta_j^{(n)}\zeta}-e^{{\rm i}\beta_j^{(n)}\zeta}),\\
\eta^{(n)}&=\delta_2^{(n)}/\delta_1^{(n)}=(e^{-{\rm
i}\beta_1^{(n)}\zeta}-e^{{\rm i}\beta_1^{(n)}\zeta})/(e^{-{\rm
i}\beta_2^{(n)}\zeta}-e^{{\rm i}\beta_2^{(n)}\zeta})
\end{cases}
\end{equation}
and
\begin{equation}\label{hchi}
\hat{\chi}^{(n)}=\chi^{(n)}+4(\delta_2^{(n)}-\delta_1^{(n)}
-\delta_1^{(n)}\delta_2^{(n)})\alpha_n^2\beta_1^{(n)}\beta_2^{(n)}/\chi^{(n)}.
\end{equation}
Here, the hyperbolic cotangent function is defined as
\[
 \coth(t)=(e^t+e^{-t})/(e^t-e^{-t}).
\]
Following the Helmholtz decomposition \eqref{chdv} again, we have
\begin{align}\label{cvre}
\hat{\boldsymbol v}(x, y)={\rm i} \sum_{n\in\mathbb{Z}}
&\begin{bmatrix}
          \alpha_n\\[5pt]
          \beta_1^{(n)}
         \end{bmatrix}
A_1^{(n)}e^{{\rm i}\bigl(\alpha_n x
+\beta_1^{(n)}\int_{b}^y\rho(\tau){\rm
d}\tau\bigr)}+\begin{bmatrix}
                     \alpha_n\\[5pt]
                     -\beta_1^{(n)}
                    \end{bmatrix}
B_1^{(n)}e^{{\rm i}\bigl(\alpha_n x
-\beta_1^{(n)}\int_{b}^y\rho(\tau){\rm d}\tau\bigr)}\notag\\
 &\begin{bmatrix}
                      \beta_2^{(n)}\\[5pt]
                      -\alpha_n
                     \end{bmatrix}
A_2^{(n)}e^{{\rm i}\bigl(\alpha_n x
+\beta_2^{(n)}\int_{b}^y\rho(\tau){\rm d}\tau\bigr)}-\begin{bmatrix}
                     \beta_2^{(n)}\\[5pt]
                     \alpha_n
                    \end{bmatrix}
B_2^{(n)}e^{{\rm i}\bigl(\alpha_n x
-\beta_2^{(n)}\int_{b}^y\rho(\tau){\rm d}\tau\bigr)}.
\end{align}

Combining \eqref{cvre} and \eqref{do}, we derive the transparent boundary
condition for the PML problem on $\Gamma$:
\[
\mathscr{D}\hat{\boldsymbol v}=\mathscr{T}^{\rm
PML}\hat{\boldsymbol v}:=\sum_{n\in\mathbb{Z}}\hat{M}^{(n)}
\hat{\boldsymbol v}^{(n)}(b)e^{{\rm i}\alpha_n x},
\]
where the matrix
\[
\hat{M}^{(n)}=
 \begin{bmatrix}
   \hat{m}_{11}^{(n)}&\hat{m}_{12}^{(n)}\\[5pt]
   \hat{m}_{21}^{(n)}&\hat{m}_{22}^{(n)}
 \end{bmatrix}.
\]
Here the entries are
\begin{align*}
  \hat{m}_{11}^{(n)}&=\frac{{\rm i}\omega^2\beta_1^{(n)}}{\hat{\chi}^{(n)}}
  +\frac{{\rm i}\omega^2\beta_1^{(n)}}{\chi^{(n)}\hat{\chi}^{(n)}}
  \Big[\varepsilon_1^{(n)}\alpha_n^2
  +(\varepsilon_1^{(n)}\eta^{(n)}+2\delta_2^{(n)})
  \beta_1^{(n)}\beta_2^{(n)}\Big],\\
  \hat{m}_{12}^{(n)}&={\rm i}\mu\alpha_n
  -\frac{{\rm i}\omega^2\alpha_{n}}{\hat{\chi}^{(n)}}
  -\frac{{\rm
i}\omega^2\alpha_{n}\beta_1^{(n)}\beta_2^{(n)}}{\chi^{(n)}\hat{\chi}^{(n)}}
  \Big[\varepsilon_1^{(n)}(1+2\delta_2^{(n)}-\eta^{(n)})
  +2\delta_2^{(n)}\Big],\\
  \hat{m}_{21}^{(n)}&=-{\rm i}\mu\alpha_n+\frac{{\rm
i}\omega^2\alpha_n}{\hat{\chi}^{(n)}}
 -\frac{{\rm
i}\omega^2\alpha_n\beta_1^{(n)}\beta_2^{(n)}}{\chi^{(n)}\hat{\chi}^{(n)}}
\Big[\varepsilon_1^{(n)}(1+2\delta_2^{(n)}-\eta^{(n)}
)+2(2\delta_1^{(n)}+2\delta_1^{(n)}\delta_2^{(n)}-\delta_2^{
(n)})\Big],\\
  \hat{m}_{22}^{(n)}&=\frac{{\rm i}\omega^2\beta_2^{(n)}}{\hat{\chi}^{(n)}}
  +\frac{{\rm i}\omega^2\beta_2^{(n)}}{\chi^{(n)}\hat{\chi}^{(n)}}
  \Big[\varepsilon_1^{(n)}\beta_1^{(n)}\beta_2^{(n)}
+(\varepsilon_1^{(n)}\eta^{(n)}+2\delta_2^{(n)})\alpha_{n}
^2\Big].
\end{align*}
Equivalently, we have the transparent boundary condition for the total field
$\hat{\boldsymbol u}$ on $\Gamma$:
\[
 \mathscr{D} \hat{\boldsymbol u}=\mathscr{T}^{\rm PML}\hat{\boldsymbol
u}+\boldsymbol{f}^{\rm PML},
\]
where $\boldsymbol{f}^{\rm PML}=\mathscr{D}\hat{\boldsymbol u}_{\rm
inc}-\mathscr{T}^{\rm PML}\hat{\boldsymbol u}_{\rm inc}$.

The PML problem can be reduced to the following boundary value problem:
\begin{equation}\label{cbvp}
 \begin{cases}
  \mu\Delta\boldsymbol{u}^{\rm PML}+(\lambda+\mu)\nabla
\nabla\cdot\boldsymbol{u}^{\rm
PML}+\omega^2\boldsymbol{u}^{\rm PML}=0\quad&\text{in}~ \Omega,\\
\boldsymbol{u}^{\rm PML}=0\quad&\text{on} ~ S,\\
\mathscr{D}\boldsymbol{u}^{\rm PML}=\mathscr{T}^{\rm PML}\boldsymbol{u}^{\rm
PML}+\boldsymbol{f}^{\rm PML}\quad&\text{on}
~ \Gamma.
 \end{cases}
\end{equation}
The weak formulation of \eqref{cbvp} is to find $\boldsymbol{u}^{\rm PML}\in
H^1_{S, \rm qp}(\Omega)^2$ such that
\begin{equation}\label{cwp}
 a^{\rm PML}(\boldsymbol{u}^{\rm PML},
\boldsymbol{v})=\langle\boldsymbol{f}^{\rm PML},
\boldsymbol{v}\rangle_{\Gamma},\quad\forall\,\boldsymbol{v}\in H^1_{S, \rm
qp}(\Omega)^2,
\end{equation}
where the sesquilinear form $a^{\rm PML}: H^1_{S, \rm qp}(\Omega)^2\times
H^1_{S, \rm qp}(\Omega)^2\to\mathbb{C}$ is defined by
\begin{align}\label{csf}
 a^{\rm PML}(\boldsymbol{u}, \boldsymbol{v})=\mu\int_\Omega
\nabla\boldsymbol{u}:\nabla\bar{\boldsymbol
v}{\rm d}\boldsymbol{x}+(\lambda+\mu)\int_\Omega (\nabla\cdot\boldsymbol{u}
)(\nabla\cdot\bar{\boldsymbol v})\,{\rm d}\boldsymbol{x}\notag\\
-\omega^2\int_\Omega \boldsymbol{u}\cdot\bar{\boldsymbol v}\, {\rm
d}\boldsymbol{x}- \langle\mathscr{T}^{\rm PML}\boldsymbol{u},
\boldsymbol{v}\rangle_{\Gamma}.
\end{align}

The following lemma establishes the relationship between the variational
problem \eqref{cwp} and the weak formulation \eqref{twp}. The proof is
straightforward based on our constructions of the transparent boundary
conditions for the PML problem. The details of the proof is omitted for
briefty.

\begin{lemm}
Any solution $\hat{\boldsymbol u}$ of the variational problem \eqref{twp}
restricted to $\Omega$ is a solution of the variational \eqref{cwp}; conversely,
any solution $\boldsymbol{u}^{\rm PML}$ of the variational problem \eqref{cwp}
can be uniquely extended to the whole domain to be a solution $\hat{\boldsymbol
u}$ of the variational problem \eqref{twp} in $D$.
\end{lemm}

\subsection{Convergence of the PML solution}

Now we turn to estimating the error between $\boldsymbol{u}^{\rm PML}$ and
$\boldsymbol{u}$. The key is to estimate the error of the boundary operators
$\mathscr{T}^{\rm PML}$ and $\mathscr{T}$.

Let
\[
\Delta^{-}_j=\min\{\Delta_j^{(n)}: n\in
U_j\},\quad \Delta^{+}_j=\min\{\Delta_j^{(n)}: n\notin U_j\}.
\]
Denote
\begin{align*}
F=&\max_{j=1, 2} \left(\frac{\Delta_j^{-}}{e^{\frac{1}{2}\Delta_j^{-}{\rm
Im}\zeta} - 1},\, \frac{\Delta_j^{+}}{e^{\frac{1}{2}\Delta_j^+{\rm
Re}\zeta}-1}\right)\\
&\quad\times \max \left\{12\kappa_2,16\kappa_2^4, 8+2\kappa_2^2,
\frac{16\kappa_2^3}{\kappa_1^2},\frac{24(16+\kappa_2^2)^2}{\kappa_1^2}\right\}.
\end{align*}
The constant $F$ will be used to control the modeling error between the
PML problem and the original scattering problem. Once the incoming plane wave
$\boldsymbol{u}_{\rm inc}$ is fixed, the quantities $\Delta_j^{-},
\Delta_j^{+}$ are fixed. Thus the constant $F$ approaches to zero
exponentially as the PML parameters ${\rm Re}\zeta$ and ${\rm Im}\zeta$ tend to
infinity. Recalling the definition of $\zeta$ in \eqref{zeta}, we know that
${\rm Re}\zeta$ and ${\rm Im}\zeta$ can be calculated by the medium property
$\rho(y)$, which is usually taken as a power function:
\[
 \rho(y)= 1+\sigma\left(\dfrac{y-b}{\delta}\right)^m \quad\text{if}~ y\geq b,
\quad m\geq 1.
\]
Thus we have
\[
 {\rm Re}\zeta=\left(1+\frac{{\rm Re}\sigma}{m+1}\right)\delta, \quad {\rm
Im}\zeta=\left(\frac{{\rm Im}\sigma}{m+1}\right)\delta.
\]
In practice, we may pick some appropriate PML parameters $\sigma$ and $\delta$
such that ${\rm Re}\zeta\geq 1$.

\begin{lemm}\label{boe}
For any $\boldsymbol{u}, \boldsymbol{v}\in H^1_{S, \rm qp}(\Omega)^2$, we have
\[
|\langle (\mathscr{T}^{\rm PML}-\mathscr{T})\boldsymbol{u},
\boldsymbol{v}\rangle_{\Gamma}|\leq \hat{F}\|\boldsymbol{u}\|_{L^2(\Gamma)^2}
\|\boldsymbol{v}\|_{L^2(\Gamma)^2},
\]
where $\hat{F}=17\omega^2 F/\kappa_1^4$.
\end{lemm}

\begin{proof}
For any $\boldsymbol{u}, \boldsymbol{v}\in H^1_{S, \rm qp}(\Omega)^2$, we have
the following Fourier series expansions:
\[
 \boldsymbol{u}(x, b)=\sum_{n\in\mathbb{Z}}\boldsymbol{u}^{(n)}(b)e^{{\rm
i}\alpha_n x},\quad
\boldsymbol{v}(x, b)=\sum_{n\in\mathbb{Z}}\boldsymbol{v}^{(n)}(b)e^{{\rm
i}\alpha_n x},
\]
which gives
\[
 \|\boldsymbol{u}\|^2_{L^2(\Gamma)^2}=\Lambda\sum_{n\in\mathbb{Z}}|\boldsymbol
{ u }^{(n)} (b)|^2 , \quad
\|\boldsymbol{v}\|^2_{L^2(\Gamma)^2}=\Lambda\sum_{n\in\mathbb{Z}}
|\boldsymbol{v} ^{(n)} (b)|^2.
\]
It follows from the orthogonality of Fourier series, the Cauchy--Schwarz
inequality, and Proposition \ref{me} that we have
\begin{align*}
&|\langle (\mathscr{T}^{\rm PML}-\mathscr{T})\boldsymbol{u},
\boldsymbol{v}\rangle_{\Gamma}|=\bigg{|}\Lambda\sum_{n\in\mathbb{Z}}\bigl(
(M^{(n)}-\hat{M}^{(n)}) \boldsymbol{u}^{(n)}(b)\bigr)\cdot\bar{\boldsymbol
v}^{(n)}(b)\bigg{|}\\
&\leq\Bigl(\Lambda\sum_{n\in\mathbb{Z}}\|M^{(n)}-\hat{M}^{(n)}\|_2^2\,
|{\boldsymbol u}^{(n)}(b)|^2\Bigr)^{1/2}\Bigl(\Lambda\sum_{n\in\mathbb{Z}}|
\boldsymbol{v}^{(n)}(b)|^2
\Bigr)^{1/2}\leq\hat{F}\|\boldsymbol{u}\|_{L^2(\Gamma)^2}\|\boldsymbol{v}\|_
{ L^2(\Gamma)^2},
\end{align*}
which completes the proof.
\end{proof}

Let $a=\min_{y}\{\boldsymbol{x}\in S\}$. Denote
$\tilde{\Omega}=\{\boldsymbol{x}\in\mathbb{R}^2: 0<x<\Lambda,\,a<y<b\}$.

\begin{lemm}\label{tr}
 For any $\boldsymbol{u}\in H^1_{S, \rm qp}(\Omega)^2$, we have
 \[
\|\boldsymbol{u}\|_{L^2(\Gamma)^2}\leq \|\boldsymbol{u}\|_{H^{1/2}(\Gamma)^2}
\leq\gamma_2\|\boldsymbol{u}\|_{H^1(\Omega)^2},
 \]
where $\gamma_2=(1+(b-a)^{-1})^{1/2}$.
\end{lemm}

\begin{proof}
First we have
\begin{align*}
 (b-a)|u(b)|^2&=\int_a^b|u(y)|^2{\rm
d}y+\int_a^b\int_y^b\frac{\rm d}{{\rm d}t}|u(t)|^2{\rm d}t{\rm d}y\\
&\leq \int_a^b|u(y)|^2{\rm d}y +(b-a)\int_a^b 2|u(y)|
|u'(y)|{\rm d}y,
\end{align*}
which gives by applying the Cauchy--Schwarz inequality that
\[
 (1+\alpha_n^2)^{1/2}|u(b)|^2\leq \gamma_2^2 (1+\alpha_n^2)\int_a^b |u(y)|^2{\rm
d}y+\int_a^b |u'(y)|^2{\rm d}y.
\]
Given $\boldsymbol{u}\in H^1_{S, \rm qp}(\Omega)^2$, we consider the zero
extension
\[
 \tilde{\boldsymbol u}=\begin{cases}
                        \boldsymbol{u}&\quad\text{in} ~ \Omega,\\
                        0 &\quad\text{in} ~ \tilde{\Omega}\setminus\bar{\Omega},
                       \end{cases}
\]
which has the Fourier series expansion
\[
 \tilde{\boldsymbol u}(x, y)=\sum_{n\in\mathbb{Z}}\tilde{\boldsymbol
u}^{(n)}(y)e^{{\rm i}\alpha_n x}\quad\text{in} ~ \tilde{\Omega}.
\]
By definitions, we have
\[
 \|\tilde{\boldsymbol u}\|^2_{H^{1/2}(\Gamma)^2}=\Lambda\sum_{n\in\mathbb{Z}}
(1+\alpha_n^2)^{1/2}|\tilde{\boldsymbol u}^{(n)} (b)|^2
\]
and
\[
 \|\tilde{\boldsymbol
u}\|^2_{H^1(\tilde{\Omega})^2}=\Lambda\sum_{n\in\mathbb{Z}} \int_a^b
(1+\alpha_n^2)|\tilde{\boldsymbol u}^{(n)}(y)|^2+|\boldsymbol{u}^{(n)'} (y)|^2
{\rm d}y.
\]
Noting $\|\boldsymbol{u}\|_{H^{1/2}(\Gamma)^2}=\|\tilde{\boldsymbol
u}\|_{H^{1/2}(\Gamma)^2}$ and
$\|\boldsymbol{u}\|_{H^1(\Omega)^2}=\|\tilde{\boldsymbol
u}\|_{H^1(\tilde{\Omega})^2}$, we complete the proof by combining the above
estimates.
\end{proof}

\begin{theo}\label{se}
 Let $\gamma_1$ and $\gamma_2$ be the constants in the inf-sup condition
\eqref{ifc} and in Lemma \ref{tr}, respectively. If
$\hat{F}\gamma^2_2<\gamma_1$, then the PML variational problem
\eqref{cwp} has a unique weak solution $\boldsymbol{u}^{\rm PML}$, which
satisfies the error estimate
\begin{align}\label{ee}
 \|\boldsymbol{u}-\boldsymbol{u}^{\rm PML}\|_\Omega:=\sup_{0\neq
\boldsymbol{v}\in H^1_{S,
\rm qp}(\Omega)^2}\frac{|a(\boldsymbol{u}-\boldsymbol{u}^{\rm PML},
\boldsymbol{v})|}{\|\boldsymbol{v}\|_{H^1(\Omega)^2}}
\leq \hat{F}\gamma_2 \|\boldsymbol{u}^{\rm PML}-\boldsymbol{u}_{\rm
inc}\|_{L^2(\Gamma)^2},
\end{align}
where $\boldsymbol{u}$ is the unique weak solution of the variational problem
\eqref{wp}.
\end{theo}

\begin{proof}
It suffices to show the coercivity of the sesquilinear form $a^{\rm PML}$
defined in \eqref{csf} in order to prove the unique solvability of the weak
problem \eqref{cwp}. Using Lemmas \ref{boe}, \ref{tr} and the assumption
$\hat{F}\gamma^2_2<\gamma_1$, we get for any $\boldsymbol{u},
\boldsymbol{v}$ in $H^1_{S, \rm qp}(\Omega)^2$ that
\begin{align*}
 |a^{\rm PML}(\boldsymbol{u}, \boldsymbol{v})|&\geq |a(\boldsymbol{u},
\boldsymbol{v})|-|\langle (\mathscr{T}^{\rm PML}-\mathscr{T})\boldsymbol{u},
\boldsymbol{v}\rangle_{\Gamma}|\\
&\geq|a(\boldsymbol{u}, \boldsymbol{v})|-\hat{F}\gamma_2^2\|\boldsymbol{u}\|_{
H^1(\Omega)^
2} \|\boldsymbol{v}\|_{H^1(\Omega)^2}\\
&\geq\bigl(\gamma_1-\hat{F}\gamma_2^2\bigr)\|\boldsymbol{u}\|_{
H^1(\Omega)^2}\|\boldsymbol{v}\|_{H^1(\Omega)^2}.
\end{align*}
It remains to show the error estimate \eqref{ee}. It follows from
\eqref{wp}--\eqref{sf} and \eqref{cwp}--\eqref{csf} that
\begin{align*}
 a(\boldsymbol{u}-\boldsymbol{u}^{\rm PML},
\boldsymbol{v})&=a(\boldsymbol{u}, \boldsymbol{v})-a(\boldsymbol{u}^{\rm PML},
\boldsymbol{v})\\
&=\langle\boldsymbol{f},
\boldsymbol{v}\rangle_{\Gamma}-\langle\boldsymbol{f}^{\rm PML},
\boldsymbol{v}\rangle_{\Gamma}+a^{\rm PML}(\boldsymbol{u}^{\rm PML},
\boldsymbol{v})-a(\boldsymbol{u}^{\rm PML}, \boldsymbol{v})\\
&=\langle (\mathscr{T}^{\rm PML}-\mathscr{T})\boldsymbol{u}_{\rm inc},
\boldsymbol{v}\rangle_{\Gamma}-\langle (\mathscr{T}^{\rm
PML}-\mathscr{T})\boldsymbol{u}^{\rm PML},
\boldsymbol{v}\rangle_{\Gamma}\\
&=\langle (\mathscr{T}-\mathscr{T}^{\rm PML})(\boldsymbol{u}^{\rm
PML}-\boldsymbol{u}_{\rm inc}), \boldsymbol{v}\rangle_{\Gamma},
\end{align*}
which completes the proof upon using Lemmas \ref{boe} and \ref{tr}.
\end{proof}

We remark that the error estimate \eqref{ee} is a posteriori in nature as it
depends only on the PML solution $\boldsymbol{u}^{\rm PML}$, which makes a
posteriori error control possible. Moreover, the PML approximation error can be
reduced exponentially by either enlarging the thickness $\delta$ of the PML
layers or enlarging the medium parameters ${\rm Re}\sigma$ and ${\rm
Im}\sigma$.

\section{Finite element approximation}

In this section, we consider the finite element approximation of the PML
problem \eqref{twp} and deduce the a posterior error estimate.

\subsection{The discrete problem}

Let $\mathcal{M}_h$ be a regular triangulation of the domain $D$. Every triangle
$T\in\mathcal{M}_h$ is considered as closed. We assume that any element $T$ must
be completely included in $\overline{\Omega^{\rm PML}}$ or $\overline{\Omega}$.
In order to introduce a finite element space whose functions are quasi-periodic
in the $x$ direction, we require that if $(0, y)$ is a node on the left
boundary, then $(\Lambda, y)$ is also a node on the right boundary, and vice
versa. Let $V_h(D)\subset H^1_{\rm qp}(D)$ be a conforming finite element space,
and $\mathring{V}_h(D)=V_h(D) \cap H^1_{0,\rm qp}(D)$.

Denote by $\Pi_h:C(\bar{D})^2\rightarrow V_h(D)^2$ the Scott--Zhang
interpolation operator, which has the following properties:
\[
  \|\boldsymbol{v}-\Pi_h\boldsymbol{v}\|_{L^2(T)^2}
  \leq Ch_{T}\|\nabla\boldsymbol{v}\|_{F(\tilde{T})},\quad
  \|\boldsymbol{v}-\Pi_h\boldsymbol{v}\|_{L^2(e)^2}
  \leq Ch_{e}^{1/2}\|\nabla\boldsymbol{v}\|_{F(\tilde{e})},
\]
where $h_T$ is the diameter of the triangle $T$, $h_e$ is the length of the
edge $e$, $\tilde T$ and $\tilde e$ are the unions of all elements which have
nonempty intersection with the element $T$ and the edge $e$, respectively, and
the Frobenius norm of the Jacobian matrix $\nabla \boldsymbol{v}$ is defined by
\[
\|\nabla\boldsymbol{v}\|_{F(G)}
=\left(\sum\limits_{j=1}^{2}\int_G|\nabla v_j|^2{\rm
d}\boldsymbol{x}\right)^{1/2}.
\]
The finite element approximation to the problem (\ref{twp}) reads as follows:
Find $\hat{\boldsymbol{u}}_h\in V_h(D)^2$
such that $\hat{\boldsymbol{u}}_h=\Pi_h\boldsymbol{u}_{\rm inc}$ on $\Gamma^{\rm
PML}$, $\hat{\boldsymbol{u}}_h=0$ on $S$, and
\begin{equation}\label{femvp}
  b_D(\hat{\boldsymbol{u}}_h,\boldsymbol{v}_h)
  =-\int_{D}\boldsymbol{g}\cdot\bar{\boldsymbol{v}}_h{\rm d}\boldsymbol{x},
  \quad\forall\ \boldsymbol{v}_h\in\mathring{V}_h(D)^2.
\end{equation}
Following the general theory in \cite{ba-73}, the existence of a
unique solution of the discrete problem (\ref{femvp}) and the finite element
convergence analysis depend on the following discrete inf-sup condition:
\begin{equation}\label{isc}
  \sup\limits_{0\neq\mathring{V}_h(D)^2}
  \frac{|b(\hat{\boldsymbol{u}}_h,\boldsymbol{v}_h)|}
  {\|\boldsymbol{v}_h\|_{H^1(D)^2}}
  \geq \gamma_0\|\hat{\boldsymbol{u}}_h\|_{H^1(D)^2},
  \quad\forall\ \hat{\boldsymbol{u}}_h\in\mathring{V}_h(D)^2,
\end{equation}
where the constant $\gamma_0>0$ is independent of the finite element mesh size.
Since the continuous problem \eqref{twp} has a unique solution by Theorem
\ref{se}, the sesquilinear form $b:H^1_{\rm qp}(D)^2\times H^1_{\rm
qp}(D)^2\rightarrow \mathbb{C}$ satisfies the continuous inf-sup condition. Then
a general argument of Schatz \cite{sch-74} implies that \eqref{isc} is valid
for sufficiently small mesh size $h<h^{*}$. Thanks to \eqref{isc}, an
appropriate a priori error estimate can be derived and the estimate depends on
the regularity of the PML solution $\boldsymbol{u}^{\rm PML}$. We assume that
the discrete problem \eqref{femvp} admits a unique
solution $\hat{\boldsymbol{u}}_h\in V_h(D)^2$, since we are interested in the a
posteriori error estimate and the associated adaptive algorithm.

Denote by $\mathcal{B}_h$ the set of all sides that do not lie on $\Gamma^{\rm
PML}$ and $S$. For any $T\in\mathcal{M}_h$, we introduce the residual
\[
  R_T:=(\mathscr{L}\hat{\boldsymbol{u}}_h+\boldsymbol{g})|_{T}=
  \begin{cases}
    \mathscr{L}(\hat{\boldsymbol{u}}_h-\boldsymbol{u}_{\rm inc})|_T&
    \text{if}\ T\in\Omega^{\rm PML},\\
    \mathscr{L}\hat{\boldsymbol{u}}_h|_T&\text{otherwise}.
  \end{cases}
\]
For any interior side $e\in\mathcal{B}_h$ which is the common side of $T_1$ and
$T_2$, we define the jump residual across $e$ as
\[
J_e=\mathscr{D}_{\boldsymbol\nu}\hat{\boldsymbol{u}}_h|_{T_1}-\mathscr{D}_{
\boldsymbol\nu} \hat{\boldsymbol{u}}_h|_{T_2},
\]
where the unit normal vector $\boldsymbol\nu$ on $e$ points from
$T_2$ to $T_1$ and the differential operator
\[
\mathscr{D}_{\boldsymbol\nu}\boldsymbol{v}=\mu\partial_{\boldsymbol\nu}
\boldsymbol{v}+(\lambda+\mu)(\nabla\cdot\boldsymbol{v})\boldsymbol\nu.
\]
Define
\[
  \Gamma_{\rm left}=\{\boldsymbol x\in\partial D : x=0\},\quad
  \Gamma_{\rm right}=\{\boldsymbol x\in\partial D : x=\Lambda\}.
\]
If $e=\Gamma_{\rm left}\cap\partial T$ for some element $T\in\mathcal{M}_h$
and $e'$ be the corresponding side on $\Gamma_{\rm right}$, which is also a side
for some element $T'$, then we define the jump residual as
\begin{align*}
  J_e=&\Big[\mu\partial_x(\hat{\boldsymbol{u}}_h|_T)
  +(\lambda+\mu)[1, 0]^\top\nabla_{\hat{\boldsymbol{x}}}
  \cdot(\hat{\boldsymbol{u}}_h|_T)\Big]-e^{-{\rm
i}\alpha\Lambda}\Big[\mu\partial_x(\hat{\boldsymbol{u}}_h|_{T'})
  +(\lambda+\mu)[1, 0]^\top\nabla_{\hat{\boldsymbol{x}}}
  \cdot(\hat{\boldsymbol{u}}_h|_{T'})\Big],\\
  J_{e'}=&e^{{\rm
i}\alpha\Lambda}\Big[\mu\partial_x(\hat{\boldsymbol{u}}_h|_T)
  +(\lambda+\mu)[1, 0]^\top\nabla_{\hat{\boldsymbol{x}}}
\cdot(\hat{\boldsymbol{u}}_h|_T)\Big]-\Big[\mu\partial_x(\hat{
\boldsymbol{u}}_h|_{T'})+(\lambda+\mu)[1, 0]^\top\nabla_{\hat{\boldsymbol{x}}}
  \cdot(\hat{\boldsymbol{u}}_h|_{T'})\Big].
\end{align*}
For any $T\in\mathcal{M}_h$, denote by $\eta_T$ the local error estimator:
\[
  \eta_T=h_T\|R_T\|_{L^2(T)^2}+\bigg(\frac{1}{2}\sum\limits_{e\subset \partial
T}h_e\|J_e\|_{L^2(e)^2}^2\bigg)^{1/2}.
\]

The following theorem is the main result of this paper.

\begin{theo}\label{thmerr}
There exists a positive constant $C$ such that the following a posteriori
error estimate holds
  \begin{align*}
    \|\boldsymbol{u}-\hat{\boldsymbol{u}}_h\|_{H^1(\Omega)^2}
    \leq&\gamma_2\hat{F}\|\hat{\boldsymbol{u}}_h-\boldsymbol{u}_{\rm
inc}\|_{L^2(\Gamma)^2}+\gamma_2 C_2\|\Pi_h{\boldsymbol{u}_{\rm
inc}}-\boldsymbol{u}_{\rm inc}\|_{L^2(\Gamma^{\rm PML})^2}\\
&+C(1+\gamma_2C_1)\bigg{(}\sum\limits_{T\in\mathcal{M}_h}\eta_T^2\bigg{)}^{1/2},
  \end{align*}
  where the constants $\hat{F}$, $\gamma_2$, and $C_j$ are defined in Lemmas
\ref{boe}, \ref{tr}, \ref{extend est}, \ref{bext}, respectively.
\end{theo}

\subsection{A posteriori error analysis}

For any $\boldsymbol{v}\in H^1_{\rm qp}(\Omega)^2$, we denote by
$\tilde{\boldsymbol v}$ the extension of $\boldsymbol{v}$ such that
$\tilde{\boldsymbol v}=\boldsymbol{v}$ in $\Omega$ and $\tilde{\boldsymbol v}$
satisfies the following boundary value problem
\begin{equation}\label{extend}
\begin{cases}
  \mu\Delta_{\hat{\boldsymbol{x}}}\bar{\tilde{\boldsymbol v}}
  +(\lambda+\mu)\nabla_{\hat{\boldsymbol{x}}}\nabla_{\hat{\boldsymbol{x}}}\cdot
  \bar{\tilde{\boldsymbol v}}+\omega^2\bar{\tilde{\boldsymbol v}}=0
  &\quad\text{in}\, \Omega^{\rm PML},\\
  \tilde{\boldsymbol v}(x,b)=\boldsymbol{v}(x,b) &\quad\text{on} ~
\Gamma,\\
  \tilde{\boldsymbol v}(x,b+\delta)=0 &\quad\text{on} ~ \Gamma^{\rm PML}.
\end{cases}
\end{equation}

\begin{lemm}\label{dp}
For any $\boldsymbol{u}$, $\boldsymbol{v}\in H^1_{\rm qp}(\Omega)^2$ we have
\[
\int_{\Gamma}\mathscr{T}^{\rm PML}\boldsymbol{u}\cdot\bar{\boldsymbol{v}}{\rm
d}x=\int_{\Gamma}\boldsymbol{u}\cdot \mathscr{D}\bar{\tilde{\boldsymbol v}}{\rm
d}x.
\]
\end{lemm}

\begin{proof}
Introduce a function $\hat{\boldsymbol{w}}\in H_{\rm qp}^1(\Omega^{\rm
PML})^2$ which satisfies
\[
\begin{cases}
\mu\Delta_{\hat{\boldsymbol{x}}}\hat{\boldsymbol{w}}+(\lambda+\mu)\nabla_{\hat{
\boldsymbol{x}}}\nabla_{\hat{\boldsymbol{x}}}\cdot\hat{\boldsymbol{w}}
+\omega^2\hat{\boldsymbol{w}}=0 &\quad\text{in} ~ \Omega^{\rm PML},\\
\hat{\boldsymbol{w}}(x,b)=\boldsymbol{u}(x,b)&\quad\text{on} ~ \Gamma,\\
\hat{\boldsymbol{w}}(x,b+\delta)=0&\quad\text{on} ~ \Gamma^{\rm PML}.
\end{cases}
\]
Using the definitions of the operators $\mathscr{T}^{\rm PML}$ and
$\mathscr{D}$, we have
\[
\mathscr{T}^{\rm PML}\boldsymbol{u}=\mathscr{D}\hat{\boldsymbol{w}}\quad
\text{on} ~ \Gamma.
\]
On the other hand, it follows from Green's formula and the extension that
  \begin{align*}
    \int_{\Gamma}\boldsymbol{u}\cdot\mathscr{D}
     \bar{\tilde{\boldsymbol v}}{\rm d}x&
     =\int_{\Gamma}\hat{\boldsymbol{w}}\cdot
\mathscr{D}\bar{\tilde{\boldsymbol v}}{\rm d}x=-\int_{\Omega^{\rm PML}}
  \Big[\mu\nabla_{\hat{\boldsymbol{x}}}\bar{\tilde{\boldsymbol v}}:
  \nabla_{\hat{\boldsymbol{x}}}\hat{\boldsymbol{w}}
+(\lambda+\mu)(\nabla_{\hat{\boldsymbol{x}}}\cdot\bar{\tilde{\boldsymbol
v }})(\nabla_{\hat{\boldsymbol{x}}}\cdot\hat{\boldsymbol{w}})
-\omega^2\bar{\tilde{\boldsymbol v}}\cdot\hat{\boldsymbol{w}}\Big]{\rm
d}\boldsymbol x \\
&=\int_{\Omega^{\rm
PML}}\Big[\mu\Delta_{\hat{\boldsymbol{x}}}\hat{\boldsymbol{w} }
  +(\lambda+\mu)\nabla_{\hat{\boldsymbol{x}}}\nabla_{\hat{\boldsymbol{x}}}\cdot
\hat{\boldsymbol{w}}+\omega^2\hat{\boldsymbol{w}}\Big]\cdot\bar{\tilde{
\boldsymbol v}}{\rm d}\boldsymbol x
+\int_{\Gamma}\mathscr{D}\hat{\boldsymbol{w}}\cdot\bar{\tilde{\boldsymbol
v}}{\rm d}x\\
&=\int_{\Gamma}\mathscr{D}\hat{\boldsymbol{w}}\cdot\bar{\tilde{\boldsymbol
v}}{\rm d}x=\int_{\Gamma}\mathscr{T}^{\rm
PML}\boldsymbol{u}\cdot\bar{\tilde{\boldsymbol v}}{\rm d}x,
\end{align*}
which completes the proof.
\end{proof}

Define $\mathring{H}^1_{\rm qp}(D)=\{v\in H^1_{\rm qp}(D): v=0 ~ \text{on} ~
\Gamma^{\rm PML}\}$. The following two lemmas are concerned with the stability
of the extension. The proofs are given in Appendix.

\begin{lemm}\label{extend est}
Let $\boldsymbol{v}\in H^1_{\rm qp}(\Omega)^2$ and $\tilde{\boldsymbol v}\in
\mathring{H}^1_{\rm qp}(D)^2$ be its extension satisfying \eqref{extend}. Then
there exists a positive constant $C_1$ such that
\[
\|\nabla\tilde{\boldsymbol v}\|_{F(\Omega^{\rm PML})}\leq
\gamma_2 C_1\|\boldsymbol{v}\|_{H^1(\Omega)^2},
\]
\end{lemm}

\begin{lemm}\label{bext}
Let $\boldsymbol{v}\in H^1_{\rm qp}(\Omega)^2$ and $\tilde{\boldsymbol v}\in
\mathring{H}^1_{\rm qp}(D)^2$ be its extension satisfying \eqref{extend}. Then
there exists a positive constant $C_2$ such that
\[
\|\mathscr{D}\tilde{\boldsymbol v}\|_{L^2(\Gamma^{\rm
PML})^2}\leq \gamma_2 C_2\|\boldsymbol{v}\|_{H^1(\Omega)^2}.
\]
\end{lemm}

For simplicity, we shall write $\tilde{\boldsymbol v}$ as
$\boldsymbol{v}$ in the rest of the paper since no confusion of the
notation is incurred.

\begin{lemm}[Error representation formula]\label{erf}
For any $\boldsymbol{v}\in H^1_{S, \rm qp}(\Omega)^2$, which is extended to be a
function in $H^1_{0, \rm qp}(D)^2$ according to
\eqref{extend}, and $\boldsymbol{v}_h\in\mathring{V}_h(D)^2$, we have
  \begin{align*}
    a(\boldsymbol{u}-\hat{\boldsymbol{u}}_h,\boldsymbol{v})
    =&-\int_D \boldsymbol{g}\cdot(\bar{\boldsymbol v}-\bar{\boldsymbol v}_h){\rm
d}\boldsymbol{x}-b_D(\hat{\boldsymbol{u}}_h,\boldsymbol{v}-\boldsymbol{v}_h)\\
    &+\int_{\Gamma}(\mathscr{T}-\mathscr{T}^{\rm PML})
    (\hat{\boldsymbol{u}}_h-\boldsymbol{u}_{\rm inc})
    \cdot\bar{\boldsymbol{v}}{\rm d}x\label{errform}+\int_{\Gamma^{\rm PML}}
    (\Pi_h\boldsymbol{u}_{\rm inc}-\boldsymbol{u}_{\rm inc})
    \cdot \mathscr{D}_{\hat{\boldsymbol{x}}}\bar{\boldsymbol v}{\rm d}x.
  \end{align*}
\end{lemm}

\begin{proof}
It follows from \eqref{sf} and \eqref{twp} that
\begin{align*}
a(\boldsymbol{u}-\hat{\boldsymbol{u}}_h,\boldsymbol{v})
&=a(\boldsymbol{u}-\hat{\boldsymbol{u}},\boldsymbol{v})
 +a(\hat{\boldsymbol{u}}-\hat{\boldsymbol{u}}_h,\boldsymbol{v})\\
    &=\int_{\Gamma}(\mathscr{T}-\mathscr{T}^{\rm
PML})(\hat{\boldsymbol{u}}-\boldsymbol{u}_{\rm
inc})\cdot\bar{\boldsymbol{v}}{\rm d}x+a^{\rm
PML}(\hat{\boldsymbol{u}}-\hat{\boldsymbol{u}}_h,\boldsymbol{v})
    -\int_{\Gamma}(\mathscr{T}-\mathscr{T}^{\rm
PML})(\hat{\boldsymbol{u}}-\hat{\boldsymbol{u}}_h)
    \cdot\bar{\boldsymbol{v}}{\rm d}x\\
    &=\int_{\Gamma}(\mathscr{T}-\mathscr{T}^{\rm
PML})(\hat{\boldsymbol{u}}_h-\boldsymbol{u}_{\rm
inc})\cdot\bar{\boldsymbol{v}}{\rm d}x
+a^{\rm PML}(\hat{\boldsymbol{u}}-\hat{\boldsymbol{u}}_h,\boldsymbol{v}).
\end{align*}
Using \eqref{sf} and Lemma \ref{dp} give
\[
a^{\rm
PML}(\hat{\boldsymbol{u}}-\hat{\boldsymbol{u}}_h,\boldsymbol{v})=b_{\Omega
}(\hat{\boldsymbol{u}}-\hat{\boldsymbol{u}}_h,\boldsymbol{v})-\int_{\Gamma}
\mathscr{T}^{\rm PML}(\hat{\boldsymbol{u}}-\hat{\boldsymbol{u}}_h)
\cdot\boldsymbol{v}{\rm
d}x=b_{\Omega}(\hat{\boldsymbol{u}}-\hat{\boldsymbol{u}}_h,
\boldsymbol{v})-\int_{\Gamma}(\hat{\boldsymbol{u}}-\hat{\boldsymbol{u}}_h)
\cdot \mathscr{D}\bar{\boldsymbol{v}}{\rm d}x.
\]
Since $\mathscr{L}\bar{\boldsymbol{v}}=0$ in $\Omega^{\rm PML}$, we deduce by
Green's formula that
\[
b_{\Omega^{\rm
PML}}(\hat{\boldsymbol{u}}-\hat{\boldsymbol{u}}_h,\boldsymbol{v})
  =-\int_{\Gamma}(\hat{\boldsymbol{u}}-\hat{\boldsymbol{u}}_h)
  \cdot \mathscr{D}\bar{\boldsymbol{v}}{\rm d}x
  +\int_{\Gamma^{\rm PML}}(\hat{\boldsymbol{u}}-\hat{\boldsymbol{u}}_h)
  \cdot \mathscr{D}_{\hat{\boldsymbol x}}\bar{\boldsymbol v}{\rm d}x.
\]
Applying \eqref{twp} and \eqref{femvp} yields
\begin{align*}
a^{\rm PML}(\hat{\boldsymbol{u}}-\hat{\boldsymbol{u}}_h,\boldsymbol{v})
    =&b_{D}(\hat{\boldsymbol{u}}-\hat{\boldsymbol{u}}_h,\boldsymbol{v})
    -\int_{\Gamma^{\rm PML}}(\hat{\boldsymbol{u}}-\hat{\boldsymbol{u}}_h)
  \cdot \mathscr{D}_{\hat{\boldsymbol x}}\bar{\boldsymbol{v}}{\rm d}x\\
    =&-\int_D \boldsymbol{g}(\bar{\boldsymbol v}-\bar{\boldsymbol v}_h){\rm
d}\boldsymbol{x}-b_D (\hat{\boldsymbol{u}}_h,\boldsymbol{v}-\boldsymbol{v}_h)
    -\int_{\Gamma^{\rm PML}}(\hat{\boldsymbol{u}}-\hat{\boldsymbol{u}}_h)
  \cdot \mathscr{D}_{\hat{\boldsymbol x}}\bar{\boldsymbol{v}}{\rm d}x,
\end{align*}
  which completes the proof.
\end{proof}

Clearly, it suffices to evaluate all the terms in the error representation
formula in order to show the posteriori error estimate in the Theorem
\ref{thmerr}. Now we present the proof as follows.

\begin{proof}
Taking $\boldsymbol{v}_h=\Pi_h\boldsymbol{v}_h\in H_{0, \rm qp}^1(D)^2$ in
Lemma \ref{erf} for the error representation formula, we have
\begin{align*}
  a(\boldsymbol{u}-\hat{\boldsymbol{u}}_h,\boldsymbol{v})
  =&-\int_D \boldsymbol{g}(\bar{\boldsymbol v}-\bar{\boldsymbol
v}_h){\rm d}\boldsymbol{x}
-b_{D}(\hat{\boldsymbol{u}}_h,\boldsymbol{v}-\boldsymbol{v}_h)\\
   &+\int_{\Gamma}(\mathscr{T}-\mathscr{T}^{\rm
PML})(\hat{\boldsymbol{u}}_h-\boldsymbol{u}_{\rm inc})\cdot\boldsymbol{v}{\rm
d}x+ \int_{\Gamma^{\rm PML}}(\Pi_h{\boldsymbol{u}}_{\rm
inc}-\boldsymbol{u}_{\rm inc})
  \cdot \mathscr{D}_{\hat{\boldsymbol x}}\bar{\boldsymbol{v}}{\rm d}x\\
=& J_1+J_2+J_3+J_4.
\end{align*}
It follows from integration by parts that
\[
J_1+J_2=\sum\limits_{T\in\mathcal{M}_h}\Big(
\int_TR_T\cdot(\bar{\boldsymbol v}-\Pi_h\bar{\boldsymbol v})
{\rm d}\boldsymbol{x}+\sum\limits_{e\subset\partial T}\frac{1}{2}\int_e J_e
\cdot(\bar{\boldsymbol v}-\Pi_h\bar{\boldsymbol v}){\rm d}x\Big),
\]
which gives after using the interpolation estimates and Lemma \ref{extend est}
that
\[
|J_1+J_2|\leq
C\sum\limits_{T\in\mathcal{M}_h}\eta_T \|\nabla\boldsymbol{v}\|_{
F(\tilde{T})}\leq C(1+\gamma_2C_1)\left(\sum\limits_{T\in\mathcal{M}_h}
\eta_T^2\right)^{1/2}
  \|\boldsymbol{v}\|_{H^1(\Omega)^2}.
\]
By Lemmas \ref{boe} and \ref{tr}, we obtain
\[
  |J_3|\leq  \hat{F}\|\hat{\boldsymbol{u}}_h-\boldsymbol{u}_{\rm
inc}\|_{L^2(\Gamma)^2} \|\boldsymbol{v}\|_{L^2(\Gamma)^2}
  \leq \gamma_2\hat{F}\|\hat{\boldsymbol{u}}_h-\boldsymbol{u}_{\rm
inc}\|_{L^2(\Gamma)^2}\|\boldsymbol{v}\|_{H^1(\Omega)^2}.
\]
Finally, it follows from Lemmas \ref{tr} and \ref{bext}  that
\begin{align*}
  |J_4|\leq & C_2 \|\Pi_h\boldsymbol{u}_{\rm inc}
  -\boldsymbol{u}_{\rm inc}\|_{L^2(\Gamma^{\rm PML})^2}
  \|\boldsymbol{v}\|_{L^2(\Gamma)^2}\\
  \leq & \gamma_2 C_2 \|\Pi_h\boldsymbol{u}_{\rm inc}
  -\boldsymbol{u}_{\rm inc}\|_{L^2(\Gamma^{\rm PML})^2}
  \|\boldsymbol{v}\|_{H^1(\Omega)^2}.
\end{align*}
The proof is completed by combining the above estimates
\end{proof}

\section{Numerical experiments}

According to the discussion in section 3, we choose the PML medium property as
the power function and need to specify the thickness $\delta$ of the layers
and the medium parameter $\sigma$. Recall from Theorem \ref{thmerr} that the a
posteriori error estimate consists of two parts: the PML error $\epsilon_{\rm
PML}$ and the finite element discretization error $\epsilon_{\rm FEM}$, where
\begin{align}
\label{epml} \epsilon_{\rm PML}&=\hat{F}\|\boldsymbol{u}^{\rm
PML}_h-\boldsymbol{u}_{\rm inc}\|_{L^2(\Gamma)^2},\\
\label{efem}\epsilon_{\rm FEM}&=\|\boldsymbol{u}^{\rm
PML}_h-\boldsymbol{u}_{\rm inc}\|_{L^2(\Gamma^{\rm PML})^2}+
\left(\sum_{T\in\mathcal{M}_h}\eta^2_T\right)^{1/2}.
\end{align}
In our implementation, we first choose $\delta$ and $\sigma$ such that
$\hat{F} \Lambda^{1/2}\leq 10^{-8}$, which makes the PML error negligible
compared with the finite element discretization error. Once the PML region and
the medium property are fixed, we use the standard finite element adaptive
strategy to modify the mesh according to the a posteriori error estimate
\eqref{efem}. For any $T\in\mathcal{M}_h$, we define the local a posteriori
error estimator
\[
 \hat{\eta}_T=\eta_T +\|\Pi_h\boldsymbol{u}_{\rm
inc}-\boldsymbol{u}_{\rm inc}\|_{L^2(\Gamma^{\rm PML}\cap\partial T)^2}.
\]
The adaptive FEM algorithm is summarized in Table \ref{alg}.

\begin{table}
\hrulefill

\begin{tabular}{ll}
 1 & Given the tolerance $\epsilon > 0, \tau\in (0,1)$;\\
 2 & Choose $\delta$ and $\sigma$ such that $\hat{F} \Lambda^{1/2}\leq
10^{-8}$;\\
 3 & Construct an initial triangulation $\mathcal{M}_h$ over $\Omega$ and
compute error estimators;\\
 4 &  While $\epsilon_h>\epsilon$ do\\
 5 & \qquad choose  $\hat{\mathcal{M}}_h\subset\mathcal{M}_h$ according to the
strategy $\eta_{\hat{\mathcal{M}}_h}>\tau\eta_{\mathcal{M}_h}$;\\
 6 & \qquad refine all the elements in $\hat{\mathcal{M}}_h$ and obtain a new
mesh denoted still by $\mathcal{M}_h$;\\
 7 & \qquad solve the discrete problem \eqref{femvp} on the new mesh
$\mathcal{M}_h$;\\
 8 & \qquad compute the corresponding error estimators;\\
 9 & End while.
\end{tabular}

\hrulefill
\caption{The adaptive FEM algorithm.}
\label{alg}
\end{table}

In the following, we present two examples to demonstrate the
competitive numerical performance of the proposed algorithm. We choose
$\lambda=1$ and $\mu=2$. The implementation of the adaptive algorithm is based
on FreeFem++-cs \cite{h-jnm12}.

{\em Example} 1. We consider the simplest periodic structure, a straight line.
In this situation, the exact solution is available, which allows us to test the
accuracy of the numerical algorithm. Assume that a plane compressional plane
wave $\boldsymbol{u}_{\rm inc}=[\sin\theta,\,-\cos\theta]^\top e^{{\rm i}(\alpha
x-\beta y)}$ is incident on the straight line $y=0$, where
$\alpha=\kappa_1\sin\theta, \beta=\kappa_1\cos\theta, \theta\in(-\pi/2,\,\pi/2)$
is the incident angle. It follows from the Navier equation, Helmholtz
decomposition, and outgoing radiation condition that we obtain the exact
solution
\[
\boldsymbol{u}(x, y)=\boldsymbol{u}_{\rm inc}(x, y)-[\alpha,\,
\beta]^\top R_1 e^{{\rm i}(\alpha x+\beta y)}-[\beta_2^{(0)}, \, -\alpha]^\top
R_2 e^{{\rm i}(\alpha x+\beta_2^{(0)}y)},
\]
where $\beta_2^{(0)}=(\kappa_2^2-\alpha^2)^{1/2}$ and
\[
R_1=\left(\frac{\alpha\sin\theta-\beta_2^{(0)}\cos\theta}{\alpha^2+\beta\beta_2^
{(0)}} \right),
\quad
R_2=\left(\frac{\alpha\cos\theta+\beta\sin\theta}{\alpha^2+\beta\beta_2^{(0)}}
\right).
\]
In our experiment, the parameters are chosen as $\theta=\pi/6$,
$\omega=2\pi$, and the domain $\Omega=(0,1)\times(0,1)$.
Figure \ref{ex1:err} shows the curves of
$\log \|\nabla(\boldsymbol{u}-\hat{\boldsymbol u}_k)\|_{F(\Omega)}$ versus
$\log N_k$ for both the a priori and the a posteriori error estimates, where
$N_k$ is the number of nodes of the mesh $\mathcal{M}_k$. The result shows that
the meshes and the associated numerical complexity are quasi-optimal for the
proposed method, i.e., $\log \|\nabla(\boldsymbol{u}-\hat{\boldsymbol
u}_k)\|_{F(\Omega)}=CN_k^{-1/2}$ is valid asymptotically.

\begin{figure}
\centering
\includegraphics[width=0.4\textwidth]{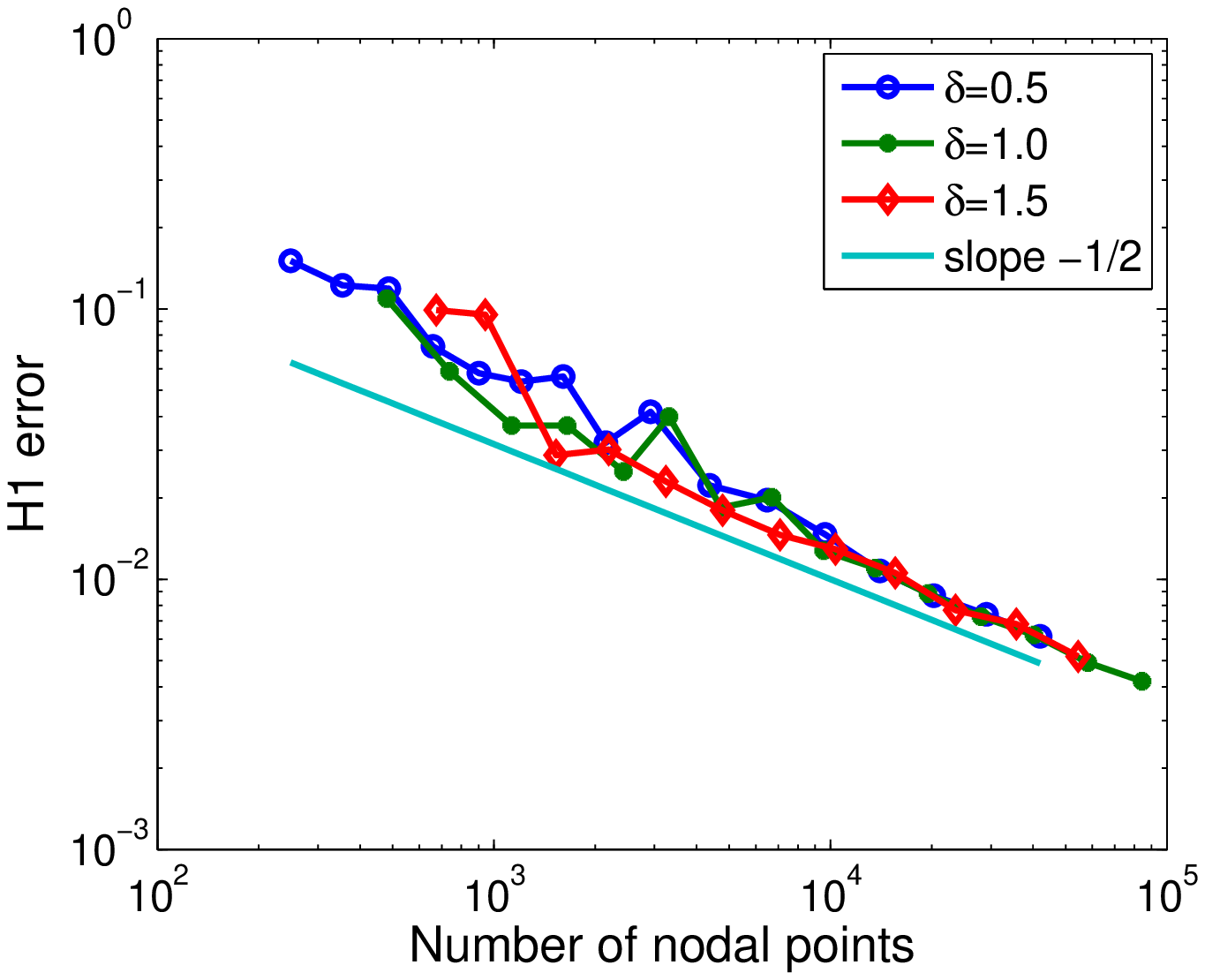}
\includegraphics[width=0.4\textwidth]{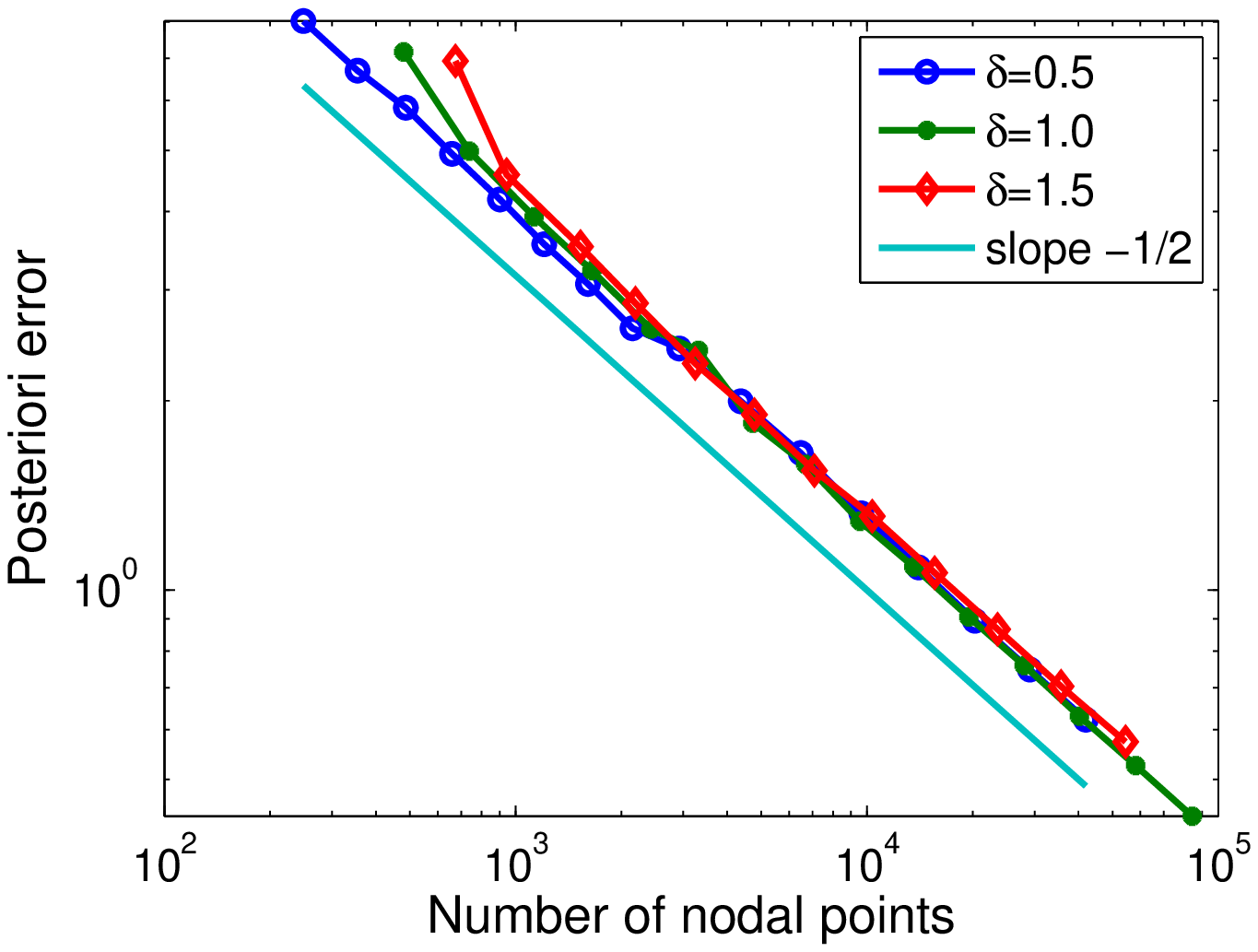}
\caption{Example 1: Quasi-optimality of the a priori (left) and a posteriori
(right) error estimates.}
\label{ex1:err}
\end{figure}

{\em Example 2}. This example is concerned with the scattering of the
compressional plane wave
$\boldsymbol{u}_{\rm inc}=[\sin\theta,\,-\cos\theta]^\top e^{{\rm i}(\alpha
x-\beta y)}$ on a grating surface with a sharp angle. The problem geometry
is shown in Figure \ref{ex2:geo}. The parameters are chosen the same as those
for Example 1. Since there is no exact solution for this example, we plot in
Figure \ref{ex2:err} the curves of $\log\|\nabla(\boldsymbol{u}-\hat{\boldsymbol
u}_k)\|_{F(\Omega)}$ versus $\log N_k$ for the a posteriori error estimate,
where $N_k$ is the number of nodes of the mesh $\mathcal{M}_k$. Again, the
result shows that the meshes and the associated numerical complexity are
quasi-optimal for the proposed method. To verify Theorem \ref{ce}, we plot
in Figure \ref{ex2:eff} the grating efficiencies and the errors of the total
efficiency for different PML thickness. Figure \ref{ex3:mesh} shows the mesh and
the amplitude of the associated solution after 6 adaptive iterations when the
grating efficiency is stabilized. The mesh has 8491 nodes. This example shows
clearly the ability of the proposed method to capture the singularity of the
solution.

\begin{figure}
\center
\includegraphics[width=0.3\textwidth]{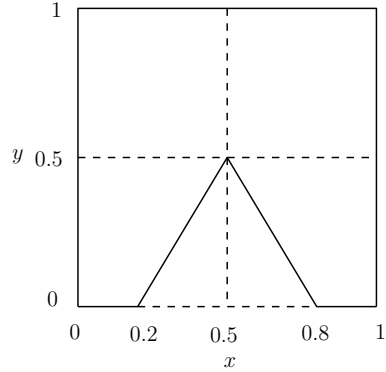}
\caption{Example 2: Geometry of the domain}
\label{ex2:geo}
\end{figure}

\begin{figure}
\center
\includegraphics[width=0.4\textwidth]{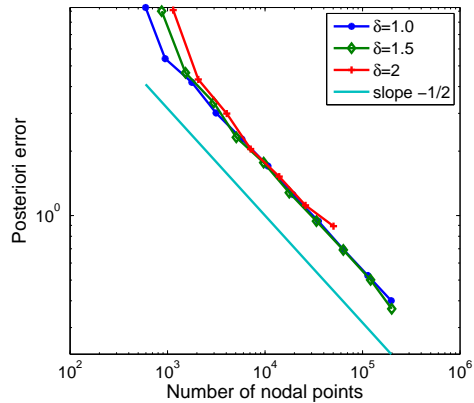}
\caption{Example 2: Quasi-optimality of the a posteriori error estimates for
different PML thickness.}
\label{ex2:err}
\end{figure}

\begin{figure}
\center
\includegraphics[width=0.4\textwidth]{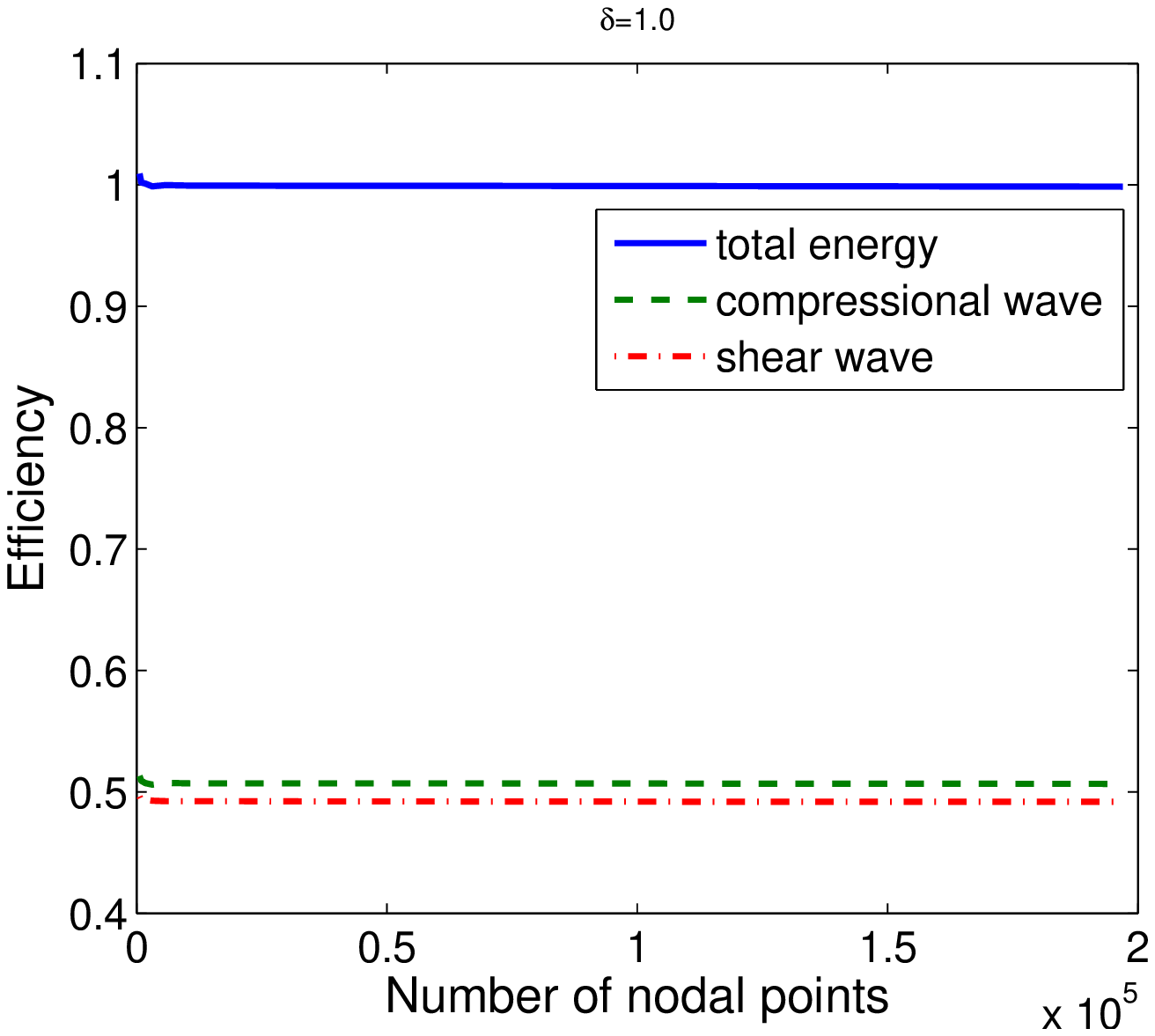}
\includegraphics[width=0.4\textwidth]{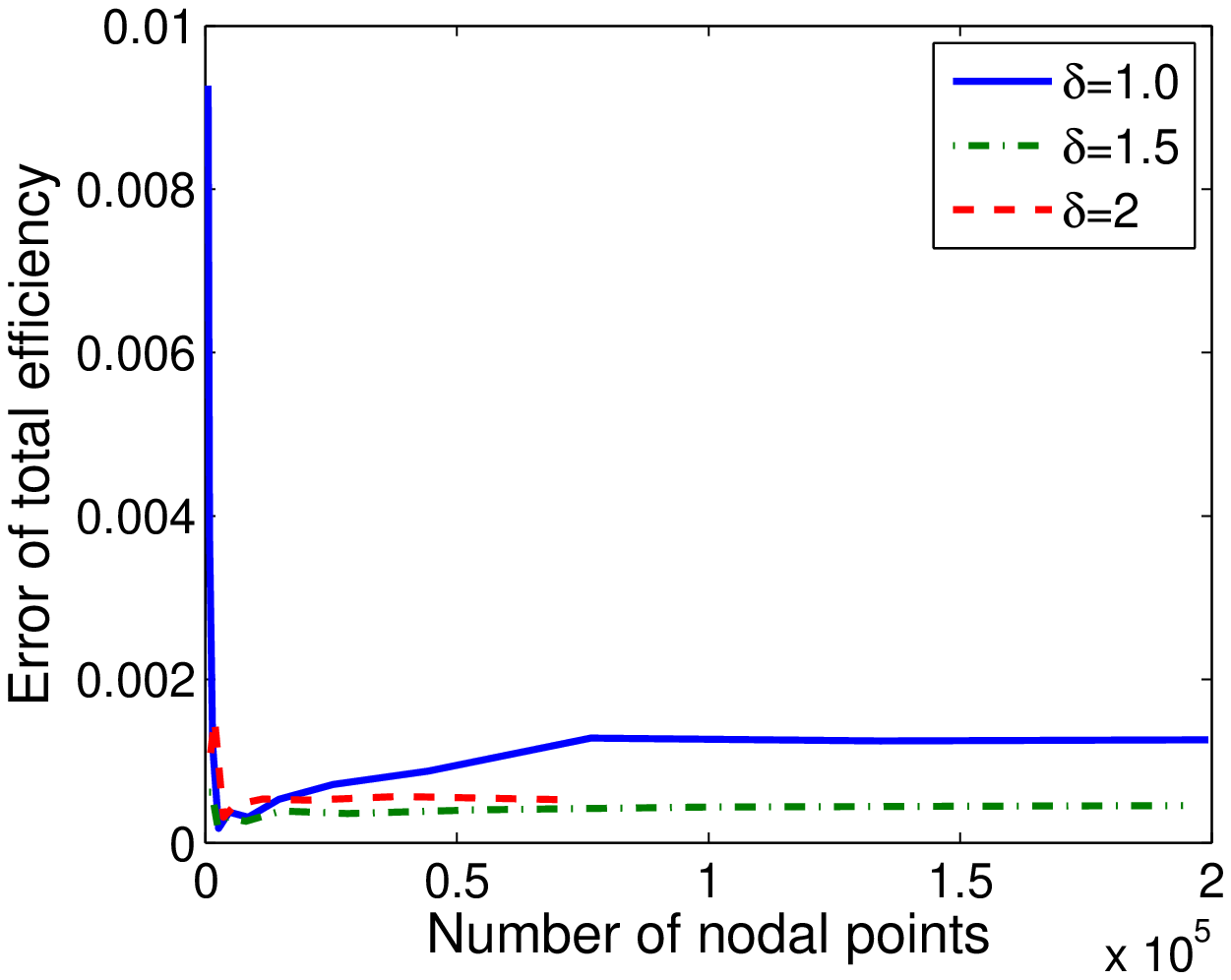}
\caption{Example 2: (left) Grating efficiency with $\delta=1.0$; (right)
Robustness of grating efficiency with respect to the thickness of PML
layers}
\label{ex2:eff}
\end{figure}

\begin{figure}
\centering
\includegraphics[width=0.25\textwidth]{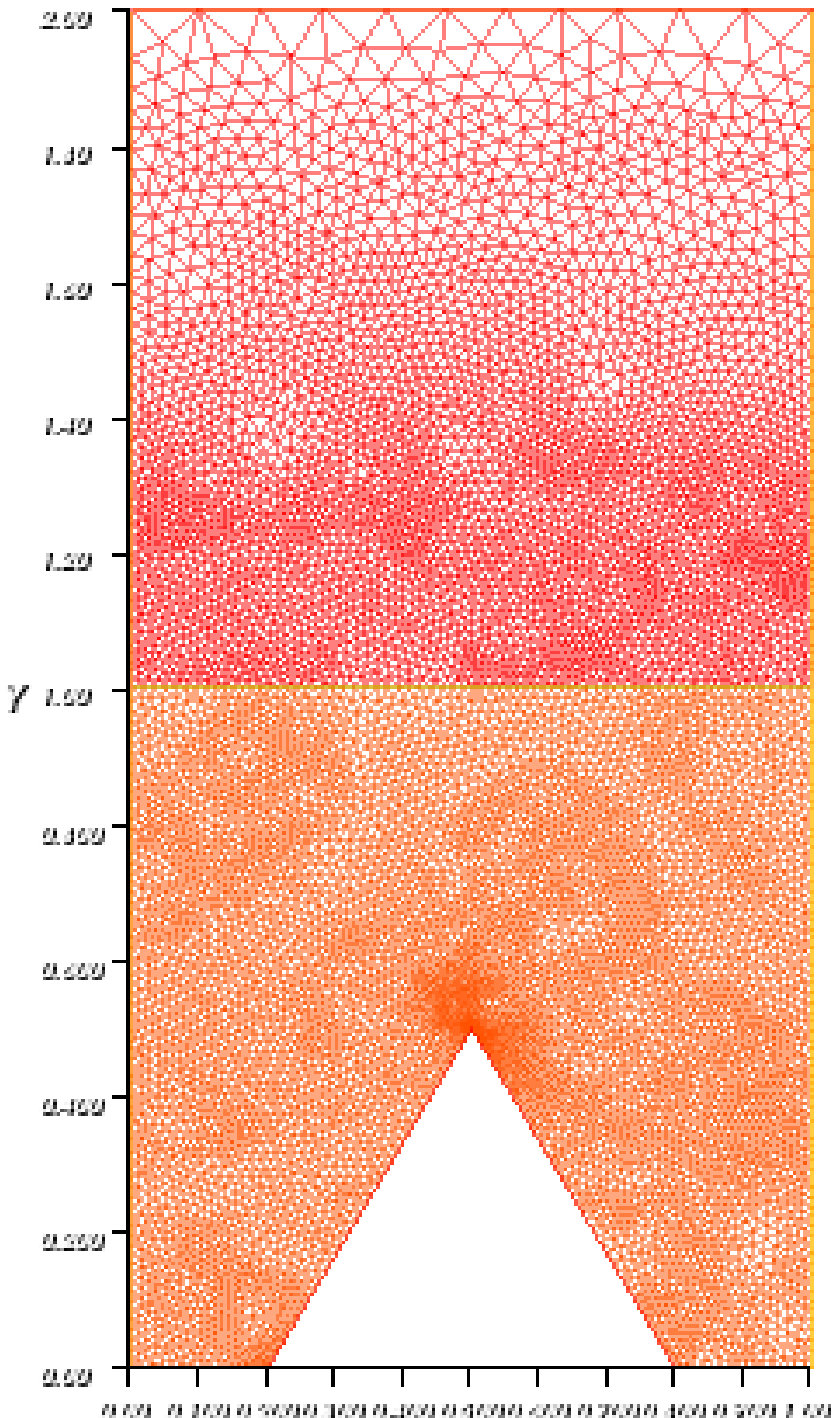}
\hspace{2.5cm}
\includegraphics[width=0.35\textwidth]{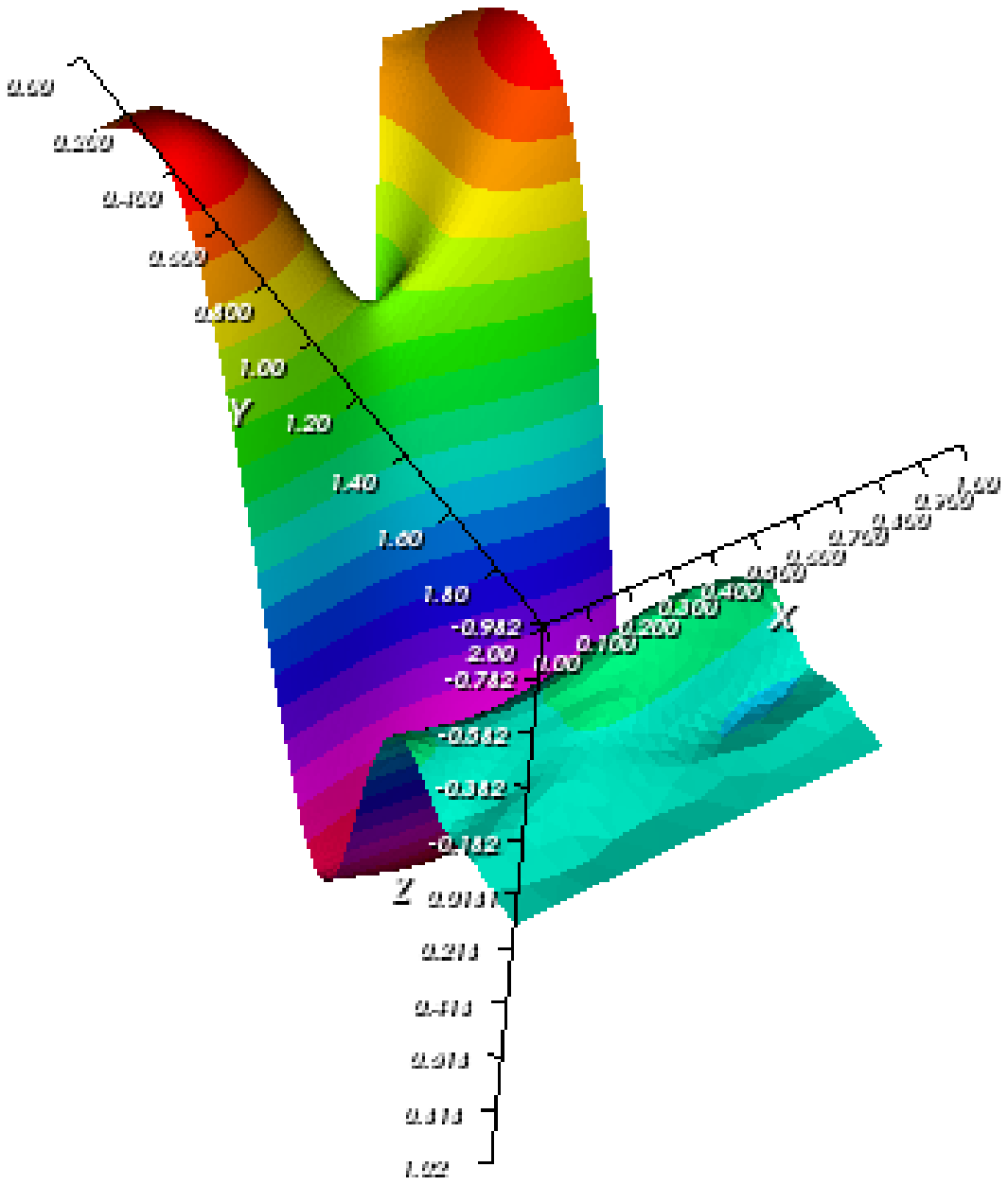}
\caption{Example 2: The mesh (left) and the surface plot of the amplitude
of the associated solution (right) after 6 adaptive iterations. The mesh has
8491 nodes.}\label{ex3:mesh}
\end{figure}

\section{Concluding remarks}

We presented an adaptive finite element method with the PML absorbing layer
technique for the elastic wave scattering problem in a periodic structure. We
showed that the truncated PML problem has a unique weak solution which
converges exponentially to the solution of the original problem by increasing
the PML parameters. We deduced the a posteriori error estimate for the PML
solution which serves as a basis for the adaptive finite element approximation.
Numerical results show that the proposed method is effective to solve the
diffractive grating problem of elastic waves. The method can be directly
applied to solve the diffraction grating problems with other interface and/or
boundary conditions. We are also currently extending the method to the
three-dimensional problem where biperiodic structures need to be considered.

\appendix

\section{Technical estimates}

In this section, we present the proofs for some technical estimates which are
used in our analysis for the error estimate between the solutions of the PML
problem and the original scattering problem.

\begin{prop}\label{chie}
For any $n\in\mathbb{Z}$, we have $\kappa_1^2<|\chi^{(n)}|<\kappa_2^2$.
\end{prop}

\begin{proof}
Recalling \eqref{chi} and \eqref{beta}, we consider three cases:

\begin{enumerate}

\item[(i)] For $n\in U_1$, $\beta_1^{(n)}=(\kappa_1^2-\alpha_n^2)^{1/2}$ and
$\beta_2^{(n)}=(\kappa_2^2-\alpha_n^2)^{1/2}$. We
have
\[
\chi^{(n)}=\alpha^2_n+\beta_1^{(n)}\beta_2^{(n)}
=\alpha_n^2+(\kappa_1^2-\alpha_n^2)^{1/2}(\kappa_2^2-\alpha_n^2)^{1/2}.
\]
Consider the function
\[
g_1(t)=t+(k_1-t)^{1/2}(k_2-t)^{1/2}, \quad 0<k_1<k_2.
\]
It is easy to know that $g_1$ is decreasing for $0<t<k_1$. Hence
\[
k_1=g_1(k_1)<g_1(t)<g_1(0)=(k_1 k_2)^{1/2},
\]
which gives $\kappa_1^2<\chi^{(n)}<\kappa_1\kappa_2$.

\item[(ii)] For $n\in U_2\setminus U_1$, $\beta_1^{(n)}={\rm
i}(\alpha_n^2-\kappa_1^2)^{1/2}, \beta_2^{(n)}=(\kappa_2^2-\alpha_n^2)^{1/2}$.
We have
\[
\chi^{(n)}=\alpha_n^2+{\rm i}(\alpha_n^2-\kappa_1^2)^{1/2}
(\kappa_2^2-\alpha_n^2)^{1/2}
\]
and
\[
|\chi^{(n)}|^2=(\kappa_1^2+\kappa_2^2)\alpha_n^2-(\kappa_1\kappa_2)^2,
\]
which gives $\kappa^2_1<|\chi^{(n)}|<\kappa_2^2$.

\item[(iii)] For $n\notin U_2$, $\beta_1^{(n)}={\rm
i}(\alpha_n^2-\kappa_1^2)^{1/2},
\beta_2^{(n)}={\rm i}(\alpha_n^2-\kappa_2^2)^{1/2}$.
We have
\[
\chi^{(n)}=\alpha_n^2-(\alpha_n^2-\kappa^2_1)^{1/2}(\alpha_n^2-\kappa^2_2)^{1/2}
.
\]
Let
\[
g_2(t)=t-(t-k_1)^{1/2}(t-k_2)^{1/2}, \quad 0<k_1<k_2.
\]
It is easy to verify that the function $g_2$ is decreasing for $t>k_2$. Hence we
have
\[
(k_1+k_2)/2=g_2(\infty)<g_2(t)<g_2(k_2)=k_2,
\]
which gives $(\kappa_1^2+\kappa_2^2)/2<\chi^{(n)}<\kappa_2^2$.
\end{enumerate}
Combining the above estimates, we get $\kappa_1^2<|\chi^{(n)}|< \kappa_2^2$ for
any $n\in\mathbb{Z}$.
\end{proof}

\begin{prop}
 The function $g_3(t)=t/(e^t-1)$ is a decreasing function for $t>0$.
\end{prop}

\begin{proof}
A simple calculation yields
\[
 g_3'(t)=\frac{(1-t)e^t-1}{(e^t-1)^2}<0,\quad t>0,
\]
which completes the proof.
\end{proof}

\begin{prop}\label{eg4}
The function $g_4(t)=t^k/ e^{(t^2-s^2)^{1/2}/2}$ satisfies $g_4(t)\leq
(s^2+k^2)^{k/2}$ for any $t>s>0, ~ k\geq 2$.
\end{prop}

\begin{proof}
Using the change of variables $\tau=(t^2-s^2)^{1/2}$, we have
\[
 \hat{g}_4(\tau)=\frac{(\tau^2+s^2)^{k/2}}{e^{\tau/2}}.
\]
Taking the derivative of $\hat{g}_4$ gives
\[
 \hat{g}'_4(\tau)=-\frac{(\tau^2-k\tau+s^2)(\tau^2+s^2)^{1/2}}{e^{\tau/2}}.
\]
\begin{enumerate}

\item[(i)] If $s\geq k/2$, then $\hat{g}_4'\leq 0$ for $\tau>0$. The
function $\hat{g}_4$ is decreasing and reaches its maximum at $\tau=0$, i.e.,
\[
 g_4(t)\leq\hat{g}_4(0)=s^k.
\]

\item[(ii)] If $s<k/2$, then
$\hat{g}'_4<0$ for $\tau\in(0, (k-(k^2-4s^2)^{1/2})/2)\cup
((k+(k^2-4s^2)^{1/2})/2, \infty)$
and $\hat{g}_4>0$ for $\tau\in ((k-(k^2-4s^2)^{1/2})/2,
(k+(k^2-4s^2)^{1/2})/2)$. Thus
$\hat{g}_4$ reaches its maximum at either $\tau_1=0$ or
$\tau_2=(k+(k^2-4s^2)^{1/2})/2$. Thus we have
\[
g_4(t)=\hat{g}_4(\tau)\leq \max\{\hat{g}_4(\tau_1),\, \hat{g}_4(\tau_2)\}\leq
(s^2+k^2)^{k/2}.
\]
\end{enumerate}
The proof is completed by combining the above estimates.
\end{proof}

\begin{prop}\label{me}
For any $n\in\mathbb{Z}$, we have $\| M^{(n)}-\hat{M}^{(n)}\|_2\leq \hat{F}$,
where $\hat{F}=17\omega^2 F/\kappa_1^4$.
\end{prop}

\begin{proof}
First, we have from \eqref{vepsn} that
\begin{align*}
 |\varepsilon_j^{(n)}|=&|\coth(-{\rm
i}\beta_j^{(n)}\zeta)-1|=\frac{2}{|e^{-2{\rm
i}\beta_j^{(n)}\zeta}-1|}\leq \frac{2}{|e^{-2{\rm i}\beta_j^{(n)}\zeta}|-1},\\
 |\delta_j^{(n)}|=&\bigg{|}\frac{e^{{\rm i}\beta_2^{(n)}\zeta}
 -e^{{\rm i}\beta_1^{(n)}\zeta}}{
 e^{-{\rm i}\beta_j^{(n)}\zeta}-e^{{\rm i}\beta_j^{(n)}\zeta}}\bigg{|}\leq
 \frac{2}{|e^{-{\rm i}\beta_j^{(n)}\zeta}|-1}
 =\frac{2}{e^{{\rm Im}\beta_j^{(n)}{\rm Re}\zeta+{\rm Re}\beta_j^{(n)}{\rm
Im}\zeta}-1},
\end{align*}
and
\[
 |\varepsilon_1^{(n)}\eta^{(n)}|
 =\bigg{|}\frac{2e^{{\rm i}\beta_1^{(n)}\zeta}}{
 e^{-{\rm i}\beta_2^{(n)}\zeta}-e^{{\rm i}\beta_2^{(n)}\zeta}}\bigg{|}\leq
 \frac{2}{|e^{-{\rm i}\beta_2^{(n)}\zeta}|-1}.
\]
Thus we can take proper PML parameters $\sigma$ and $\delta$ such that
$|\delta_j^{(n)}|<1$ for any $n\in \mathbb{Z}$.

Next we consider three cases:
\begin{enumerate}

\item[(i)] For $n\in U_1$, we have $\beta_1^{(n)}=\Delta_1^{(n)},
\beta_2^{(n)}=\Delta_2^{(n)}$. Using the facts that $\Delta_j^{(n)}\geq
\Delta_j^{-}$ for $n\in U_1$ and the function $g_3$ is decreasing for
$t>0$, we obtain from \eqref{chi} and \eqref{hchi} that
\[
 |\hat{\chi}^{(n)}-\chi^{(n)}|
 \leq\frac{24\Delta_1^{(n)}\Delta_2^{(n)}}{|e^{-{\rm i}\beta_1^{(n)}\zeta}|-1}
 \leq\frac{12\kappa_2\Delta_1^{-}}{e^{\frac{1}{2}\Delta_1^{-}{\rm Im}\zeta}-1}
 \leq F,
\]
and
\begin{align*}
 \max\big\{&|\alpha_n(\hat{\chi}^{(n)}-\chi^{(n)})|,\,
|\beta_1^{(n)}(\hat{\chi}^{(n)}-\chi^{(n)})|,\,
 |\beta_2^{(n)}(\hat{\chi}^{(n)}-\chi^{(n)})|\big\}\\
 \leq &\frac{24\kappa_2^2\Delta_1^{(n)}}{|e^{-{\rm i}\beta_1^{(n)}\zeta}|-1}
 \leq\frac{12\kappa_2^2\Delta_1^{-}}{e^{\frac{1}{2}\Delta_1^{-}{\rm Im}\zeta}-1}
 \leq F,
\end{align*}

\begin{align*}
\max\big\{&|\varepsilon_1\alpha_n^2\beta_1^{(n)}|,\,
 |\varepsilon_1\alpha_n\beta_1^{(n)}\beta_2^{(n)}|,\,
 |\varepsilon_1\beta_1^{(n)}(\beta_2^{(n)})^2|,\,
 |\delta_1^{(n)}\alpha_n\beta_1^{(n)}\beta_2^{(n)}|\big\}\\
 \leq &\frac{2\kappa_2^2\Delta_1^{(n)}}{e^{2\Delta_1^{(n)}{\rm Im}\zeta}-1}
 \leq \frac{\kappa_2^2\Delta_1^{-}}{e^{\Delta_1^{-}{\rm Im}\zeta}-1}\leq F,
\end{align*}
\begin{align*}
\max\Big\{&|\varepsilon_1^{(n)}\eta^{(n)}\alpha_n\beta_1^{(n)
} \beta_2^{(n)}|,\, |\varepsilon_1^{(n)}\eta^{(n)}
\alpha_n^2\beta_2^{(n)}|,\, |\delta_2^{(n)}\alpha_n^2\beta_2^{(n)}|,\\
 &|\varepsilon_1^{(n)}\eta^{(n)}(\beta_1^{(n)}
)^2\beta_2^{(n)}|,\, |\delta_2^{(n)}\alpha_n\beta_1^{(n)}\beta_2^{(n)}|,\,
 |\delta_2^{(n)}(\beta_1^{(n)})^2\beta_2^{(n)}|\Big\}\\
 \leq &\frac{2\kappa_2^2\Delta_2^{(n)}}{e^{\Delta_2^{(n)}{\rm Im}\zeta}-1}
\leq\frac{\kappa_2^2\Delta_2^{-}}{e^{\frac{1}{2}\Delta_2^{-}{\rm
Im}\zeta}-1}\leq F.
\end{align*}

\item[(ii)] For $n\in U_2\setminus U_1$, we have $\beta_1^{(n)}={\rm
i}\Delta_1{(n)}, \beta_2^{(n)}=\Delta_2^{(n)}$. Using the facts that
$\Delta_1{(n)}\geq \Delta_1^{+}, \Delta_2^{(n)}\geq \Delta_2^{-}$ for $n\in
U_2\setminus U_1$ and the function $g_3$ is decreasing for $t>0$ again, we get
\begin{align*}
|\hat{\chi}^{(n)}-\chi^{(n)}|
 \leq&\frac{16\kappa_2^3}{\kappa_1^2}\frac{\Delta_2^{(n)}}
 {e^{\Delta_2^{(n)}{\rm Im}\zeta}-1}
 +\frac{8\kappa_2^3}{\kappa_1^2}\frac{\Delta_1^{(n)}}{e^{\Delta_1^{(n)}{\rm
Re}\zeta}-1}\\
\leq&\frac{8\kappa_2^3}{\kappa_1^2}\frac{\Delta_2^{-}}
{e^{\frac{1}{2}\Delta_2^{-}{\rm Im}\zeta}-1}
 +\frac{4\kappa_2^3}{\kappa_1^2}\frac{\Delta_1^{+}}
 {e^{\frac{1}{2}\Delta_1^{+}{\rm Re}\zeta}-1}
  \leq F,
\end{align*}
and
\begin{align*}
 \max\Big\{&|\alpha_n(\hat{\chi}^{(n)}-\chi^{(n)})|,\,
 |\beta_1^{(n)}(\hat{\chi}^{(n)}-\chi^{(n)})|,\,
 |\beta_2^{(n)}(\hat{\chi}^{(n)}-\chi^{(n)})|\Big\}\\
 \leq&\frac{16\kappa_2^4\Delta_2^{(n)}}{e^{\Delta_2^{(n)}{\rm Im}\zeta}-1}
 +\frac{8\kappa_2^4\Delta_1^{(n)}}{e^{\Delta_1^{(n)}{\rm Re}\zeta}-1}
\leq\frac{8\kappa_2^4\Delta_2^{-}}{e^{\frac{1}{2}\Delta_2^{-}{\rm Im}\zeta}-1}
 +\frac{4\kappa_2^4\Delta_1^{+}}{e^{\frac{1}{2}\Delta_1^{+}{\rm Re}\zeta}-1}
  \leq F,
\end{align*}
\begin{align*}
 \max\Big\{&|\varepsilon_1\alpha_n^2\beta_1^{(n)}|,\,
 |\varepsilon_1\alpha_n\beta_1^{(n)}\beta_2^{(n)}|,\,
 |\varepsilon_1\beta_1^{(n)}(\beta_2^{(n)})^2|,\,
 |\delta_1^{(n)}\alpha_n\beta_1^{(n)}\beta_2^{(n)}|\}\\
 \leq&\frac{2\kappa_2^2\Delta_1^{(n)}}{e^{2\Delta_1^{(n)}{\rm
Re}\zeta}-1}\leq \frac{\kappa_2^2\Delta_1^{-}}{e^{\Delta_1^{-}{\rm
Re}\zeta}-1}\leq \frac{\kappa_2^2\Delta_1^{-}}{e^{\frac{1}{2}\Delta_1^{-}{\rm
Re}\zeta}-1}\leq F,
\end{align*}
\begin{align*}
 \max\Bigl\{&|\varepsilon_1^{(n)}\eta^{(n)}\alpha_n\beta_1^{(n)}\beta_2^{
(n)}|,\,|\varepsilon_1^{(n)}\eta^{(n)}(\alpha_n)^2\beta_2^{(n)}|,\,
  |\delta_2^{(n)}(\alpha_n)^2\beta_2^{(n)}|,\\
 &|\varepsilon_1^{(n)}\eta^{(n)}(\beta_1^{(n)})^2\beta_2^{(n)
}|,\,|\delta_2^{(n)}\alpha_n\beta_1^{(n)}\beta_2^{(n)}|,\,
  |\delta_2^{(n)}(\beta_1^{(n)})^2\beta_2^{(n)}|\Big\}\\
  \leq&\frac{2\kappa_2^2
\Delta_2^{(n)}}{e^{\Delta_2^{(n)}{\rm Im}\zeta}-1}
\leq\frac{\kappa_2^2\Delta_2^{-}}{e^{\frac{1}{2}\Delta_2^{-}{\rm
Im}\zeta}-1}\leq F,
\end{align*}

\item[(iii)] For $n\notin U_2$, we have $\beta_1^{(n)}={\rm i}\Delta_1^{(n)},
\beta_2^{(n)}={\rm i}\Delta_2^{(n)}$, and $\Delta_1^{(n)}>\Delta_2^{(n)}$.
Noting ${\rm Re}\zeta\geq 1$, we obtain
\begin{align*}
|\hat{\chi}^{(n)}-\chi^{(n)}|&\leq\frac{24}{\kappa_1^2}\frac{|\alpha_n|^3
 \Delta_2^{(n)}}{e^{\Delta_2^{(n)}{\rm Re}\zeta}-1}
 \leq\frac{24}{\kappa_1^2}\frac{|\alpha_n|^3}
 {e^{\frac{1}{2}\Delta_2^{(n)}}}\frac{\Delta_2^{(n)}}
 {e^{\frac{1}{2}\Delta_2^{(n)}{\rm Re}\zeta}-1}\\
&\leq\frac{24(9+\kappa_2^2)^{3/2}}{\kappa_1^2}\frac{
 \Delta_1^{+}}{e^{\frac{1}{2}\Delta_1^{+}{\rm Re}\zeta}-1}\leq F,
 \end{align*}
 and
 \begin{align*}
 \max\Big\{&|\alpha_n(\hat{\chi}^{(n)}-\chi^{(n)})|,\,
 |\beta_1^{(n)}(\hat{\chi}^{(n)}-\chi^{(n)})|,\,
 |\beta_2^{(n)}(\hat{\chi}^{(n)}-\chi^{(n)})|\Bigr\}\\
 \leq&\frac{24}{\kappa_1^2}\frac{|\alpha_n|^4
 \Delta_2^{(n)}}{e^{\Delta_2^{(n)}{\rm Re}\zeta}-1}
 \leq\frac{24}{\kappa_1^2}\frac{|\alpha_n|^4}
 {e^{\frac{1}{2}\Delta_2^{(n)}}}\frac{\Delta_2^{(n)}}
 {e^{\frac{1}{2}\Delta_2^{(n)}{\rm
Re}\zeta}-1}\leq\frac{24(16+\kappa_2^2)^{2}}{\kappa_1^2}\frac{
 \Delta_1^{+}}{e^{\frac{1}{2}\Delta_1^{+}{\rm Re}\zeta}-1}\leq F,
\end{align*}
\begin{align*}
 \max\Big\{&|\varepsilon_1\alpha_n^2\beta_1^{(n)}|,\,
 |\varepsilon_1\alpha_n\beta_1^{(n)}\beta_2^{(n)}|,\,
 |\varepsilon_1\beta_1^{(n)}(\beta_2^{(n)})^2|,\,
 |\delta_1^{(n)}\alpha_n\beta_1^{(n)}\beta_2^{(n)}|\Big\}\\
 \leq& \frac{2\alpha_n^2\Delta_1^{(n)}}{e^{2\Delta_1^{(n)}{\rm Re}\zeta}-1}
 \leq\frac{2\alpha_n^2}{e^{\Delta_1^{(n)}}}
 \frac{\Delta_1^{(n)}}{e^{\Delta_1^{(n)}{\rm Re}\zeta}-1}
 \leq 2(4+\kappa_1^2)\frac{\Delta_1^{+}}{e^{\Delta_1^{+}{\rm Re}\zeta}-1}\leq F,
\end{align*}
\begin{align*}
  \max\Big\{&|\varepsilon_1^{(n)}\eta^{(n)}
\alpha_n\beta_1^{(n)}\beta_2^{(n)}|,
  |\varepsilon_1^{(n)}\eta^{(n)}(\alpha_n)^2\beta_2^{(n)}|,
  |\delta_2^{(n)}(\alpha_n)^2\beta_2^{(n)}|,\\
&|\varepsilon_1^{(n)}\eta^{(n)}(\beta_1^{(n)})^2\beta_2^{(n)}
|,\,|\delta_2^{(n)}\alpha_n\beta_1^{(n)}\beta_2^{(n)}|,
  |\delta_2^{(n)}(\beta_1^{(n)})^2\beta_2^{(n)}|  \Big\}\\
\leq&\frac{2\alpha_n^2\Delta_2^{(n)}}{e^{\Delta_2^{(n)}{\rm Re}\zeta}-1}
\leq\frac{2\alpha_n^2}{e^{\frac{1}{2}\Delta_2^{(n)}}}
\frac{\Delta_2^{(n)}}{e^{\frac{1}{2}\Delta_2^{(n)}{\rm Re}\zeta}-1}
\leq2(4+k_2^2)\frac{\Delta_2^{+}}{e^{\frac{1}{2}\Delta_2^{+}{\rm
Re}\zeta}-1}\leq F,
\end{align*}
where we have used the estimate for $g_4$ and the facts that $\Delta_j^{(n)}\geq
\Delta_j^{+}$ for $n\notin U_2$ and $g_3$ is a decreasing function.
\end{enumerate}
It follows from Proposition \ref{chie} and the estimate
$|\hat{\chi}^{(n)}-\chi^{(n)}|\leq F$ that $\kappa_1^2-F\leq
|\hat{\chi}^{(n)}|\leq \kappa_2^2+F.$ Again, we may choose some proper PML
parameters $\sigma$ and $\delta$ such that $F\leq \kappa_1^2/2$, which gives
$|\hat{\chi}^{(n)}|\geq \kappa_1^2/2$.

Last, using the matrix norm and combining all the above estimates, we get
\begin{align*}
\|M^{(n)}&-\hat{M}^{(n)}\|_2^2
 \leq\frac{4\omega^4}{\kappa_1^8}\Big(
 |\beta_1^{(n)}(\hat{\chi}^{(n)}-\chi^{(n)})|^2
+2|\alpha_n(\hat{\chi}^{(n)}-\chi^{(n)})|^2
 +|\beta_2^{(n)}(\hat{\chi}^{(n)}-\chi^{(n)})|^2\\
&+|\varepsilon_1^{(n)}\alpha_n^2\beta_1^{(n)}|^2
+|\varepsilon_1^{(n)}\beta_1^{(n)}(\beta_2^{(n)})^2|^2
+10|\varepsilon_1^{(n)}\alpha_n\beta_1^{(n)}\beta_2^{(n)}|^2
+|\varepsilon_1^{(n)}\eta^{(n)}(\beta_1^{(n)})^2\beta_2^{(n)}
|^2\\
& +2|\varepsilon_1^{(n)}\eta^{(n)}
\alpha_n\beta_1^{(n)}\beta_2^{(n)}|^2
  +|\varepsilon_1^{(n)}\eta^{(n)}\alpha_n^2\beta_2^{(n)}|^2
 +4|\delta_2^{(n)}(\beta_1^{(n)})^2\beta_2^{(n)}|^2\\
&+16|\delta_1^{(n)}\alpha_n\beta_1^{(n)}\beta_2^{(n)}|^2
+4|\delta_2^{(n)}\alpha_n^2\beta_2^{(n)}|^2
+24|\delta_2^{(n)}\alpha_n\beta_1^{(n)}\beta_2^{(n)}|^2\Big)
\leq \frac{272\omega^4}{\kappa_1^8}F^2,
\end{align*}
which completes the proof.
\end{proof}

\section{Proof of Lemma \ref{extend est}}

Let $\boldsymbol{w}=\bar{\tilde{\boldsymbol v}}$. The problem \eqref{extend}
can be written as
\begin{equation}\label{AB-s1}
\begin{cases}
  \mu\Delta_{\hat{\boldsymbol{x}}}\boldsymbol{w}
  +(\lambda+\mu)\nabla_{\hat{\boldsymbol{x}}}\nabla_{\hat{\boldsymbol{x}}}\cdot
  \boldsymbol{w}+\omega^2\boldsymbol{w}=0  &\quad\text{in}\, \Omega^{\rm PML},\\
  \boldsymbol{w}(x,b)=\bar{\boldsymbol v}(x,b) &\quad\text{on} ~
\Gamma,\\
  \boldsymbol{w}(x,b+\delta)=0 &\quad\text{on} ~ \Gamma^{\rm PML}.
\end{cases}
\end{equation}
We introduce the Helmholtz decomposition to the solution of \eqref{AB-s1}:
\begin{equation}\label{AB-s2}
 \boldsymbol w=\nabla_{\hat{\boldsymbol x}}\psi_1+{\bf
curl}_{\hat{\boldsymbol x}}\psi_2,
\end{equation}
where $\psi_j(\hat{\boldsymbol{x}})$ satisfies the Helmholtz equation
\begin{equation}\label{AB-s3}
 \Delta_{\hat{\boldsymbol{x}}}\psi_j +\kappa_j^2\psi_j=0.
\end{equation}
Due to the quasi-periodicity of the solution, we have the Fourier series
expansion
\begin{equation}\label{AB-s4}
 \psi_j(x, y)=\sum_{n\in\mathbb{Z}}\psi_j^{(n)}(y)e^{-{\rm i}\alpha_n
x}.
\end{equation}
Substituting \eqref{AB-s4} into \eqref{AB-s3} yields
\begin{equation}\label{AB-s5}
 \rho^{-1}\frac{\rm d}{{\rm d}y}\Bigl(\rho^{-1}\frac{{\rm
d}}{{\rm d}y}\psi_j^{(n)}(y)\Bigr)+(\beta_j^{(n)})^2\psi_j^{(n)}(y)=0.
\end{equation}
The general solutions of \eqref{AB-s5} is
\begin{align*}
\psi_j^{(n)}(y)=\tilde{A}_j^{(n)} e^{{\rm
i}\beta_j^{(n)}\int_{b}^y\rho(\tau){\rm d}\tau} + \tilde{B}_j^{(n)}e^{-{\rm
i}\beta_j^{(n)}\int_{b}^y \rho(\tau){\rm d}\tau}.
\end{align*}
It follows from \eqref{AB-s2} that the coefficients $\tilde{A}_j^{(n)}$ and
$\tilde{B}_j^{(n)}$ can be uniquely determined by solving the following linear
equations
\begin{equation}\label{AB-s6}
\begin{bmatrix}
 -\alpha_n & -\alpha_n & \beta_2^{(n)} & -\beta_2^{(n)}\\[5pt]
 \beta_1^{(n)} &-\beta_1^{(n)} & \alpha_n & \alpha_n\\[5pt]
 -\alpha_n e^{{\rm i}\beta_1^{(n)}\zeta} & -\alpha_n e^{-{\rm
i}\beta_1^{(n)}\zeta}&
 \beta_2^{(n)}e^{{\rm i}\beta_2^{(n)}\zeta} & -\beta_2^{(n)}e^{-{\rm
i}\beta_2^{(n)}\zeta}\\[5pt]
 \beta_1^{(n)}e^{{\rm i}\beta_1^{(n)}\zeta} & -\beta_1^{(n)}e^{-{\rm
i}\beta_1^{(n)}\zeta}&\alpha_n e^{{\rm i}\beta_2^{(n)}\zeta} &
\alpha_n e^{-{\rm i}\beta_2^{(n)}\zeta}
\end{bmatrix}
\begin{bmatrix}
 \tilde{A}_1^{(n)}\\[5pt]
 \tilde{B}_1^{(n)}\\[5pt]
 \tilde{A}_2^{(n)}\\[5pt]
 \tilde{B}_2^{(n)}
\end{bmatrix}
=\begin{bmatrix}
  -{\rm i}\bar{v}_1^{(n)}(b)\\[5pt]
  -{\rm i}\bar{v}_2^{(n)}(b)\\[5pt]
  0\\[5pt]
  0
 \end{bmatrix}.
\end{equation}
A straightforward calculation yields the solution of \eqref{AB-s6}:
\begin{align*}
  \tilde{A}_1^{(n)}=&\frac{{\rm i}}{2\chi^{(n)}\hat{\chi}^{(n)}}
  \Big\{-\chi^{(n)}(\varepsilon_1^{(n)}+2)
  (-\alpha_n\bar{v}_1^{(n)}(b)+\beta_2^{(n)}\bar{v}_2^{(n)}(b))\\
  &+2\beta_2^{(n)}(\varepsilon_1^{(n)}
  +2\delta_1^{(n)})(1+\delta_2^{(n)}
  -\eta^{(n)})(-\alpha_n\beta_1^{(n)}\bar{v}_1^{(n)}(b)
  +\alpha_n^2\bar{v}_2^{(n)}(b))\Big\},\\
  \tilde{B}_1^{(n)}=&\frac{{\rm i}}{2\chi^{(n)}\hat{\chi}^{(n)}}
  \Big\{\chi^{(n)}\varepsilon_1^{(n)}
  (-\alpha_n\bar{v}_1^{(n)}(b)-\beta_2^{(n)}\bar{v}_2^{(n)}(b))\\
  &+2(\varepsilon_1^{(n)}\delta_2^{(n)}
  +2(\delta_1^{(n)}+\delta_1^{(n)}\delta_2^{(n)})
  (-\alpha_n\beta_1^{(n)}\beta_2^{(n)}\bar{v}_1^{(n)}(b)
  -\alpha_n^2\beta_2^{(n)}\bar{v}_2^{(n)}(b))\Big\},\\
  \tilde{A}_2^{(n)}=&\frac{{\rm i}}{2\chi^{(n)}\hat{\chi}^{(n)}}
  \Big\{\chi^{(n)}[\varepsilon_1^{(n)}\eta^{(n)}
-2(\varepsilon_1^{(n)}+1)(1+\delta_2^{(n)})]
  (\beta_1^{(n)}\bar{v}_1^{(n)}(b)+\alpha_n\bar{v}_2^{(n)}(b))\\
&+2\varepsilon_1^{(n)}(1+\delta_2^{(n)}-\eta^{(n)})((\beta_1^{(n)})^2\beta_2^{
(n)}\bar{v}_1^{(n)}(b)+\alpha_n^3\bar{v}_2^{(n)}(b))\Big\},\\
  \tilde{B}_2^{(n)}=&\frac{{\rm i}}{2\chi^{(n)}\hat{\chi}^{(n)}}
  \Big\{\chi^{(n)}[2\delta_2^{(n)}(\varepsilon_1^{(n)}+1)
  -\varepsilon_1^{(n)}\eta^{(n)}]
  (\beta_1^{(n)}\bar{v}_1^{(n)}(b)-\alpha_n\bar{v}_2^{(n)}(b))\\
  &-2\delta_2^{(n)}(\varepsilon_1^{(n)}+2)
  ((\beta_1^{(n)})^2\beta_2^{(n)}\bar{v}_1^{(n)}(b)
  -\alpha_n^3\bar{v}_2^{(n)}(b))\Big\}.
\end{align*}
Noting $\tilde{\boldsymbol v}=\bar{\boldsymbol w}$ and using the Helmholtz
decomposition \eqref{AB-s2} again, we obtain
\begin{align}\label{AB-s7}
\tilde{\boldsymbol v}(x, y)={\rm i} \sum_{n\in\mathbb{Z}}
&\begin{bmatrix}
          \alpha_n\\[5pt]
          -\bar{\beta}_1^{(n)}
         \end{bmatrix}
\bar{\tilde{A}}_1^{(n)}e^{{\rm i}\bigl(\alpha_n x
-\bar{\beta}_1^{(n)}\int_{b}^y\bar{\rho}(\tau){\rm
d}\tau\bigr)}+\begin{bmatrix}
                     \alpha_n\\[5pt]
                     \bar{\beta}_1^{(n)}
                    \end{bmatrix}
\bar{\tilde{B}}_1^{(n)}e^{{\rm i}\bigl(\alpha_n x
+\bar{\beta}_1^{(n)}\int_{b}^y\bar{\rho}(\tau){\rm d}\tau\bigr)}\notag\\
- &\begin{bmatrix}
                      \bar{\beta}_2^{(n)}\\[5pt]
                      \alpha_n
                     \end{bmatrix}
\bar{\tilde{A}}_2^{(n)}e^{{\rm i}\bigl(\alpha_n x
-\bar{\beta}_2^{(n)}\int_{b}^y\bar{\rho}(\tau){\rm
d}\tau\bigr)}+\begin{bmatrix}
                     \bar{\beta}_2^{(n)}\\[5pt]
                     -\alpha_n
                    \end{bmatrix}
\bar{\tilde{B}}_2^{(n)}e^{{\rm i}\bigl(\alpha_n x
+\bar{\beta}_2^{(n)}\int_{b}^y\bar{\rho}(\tau){\rm d}\tau\bigr)}.
\end{align}

Using the orthogonality of the Fourier modes in \eqref{AB-s7}, we have
\begin{align*}
\int_0^{\Lambda} &\big(|\partial_x\tilde{v}_1|^2+|\partial_x\tilde{v}_2|^2+
    |\partial_y\tilde{v}_1|^2+|\partial_y\tilde{v}_2|^2 \big){\rm d}x
    \leq 2\Lambda\sum\limits_{n\in\mathbb{Z}}\Big[
     |\alpha_n^2\bar{\tilde{A}}_1^{(n)}e^{-{\rm i}\beta_1^{(n)}\hat{y}}|^2
    +|\alpha_n^2\bar{\tilde{B}}_1^{(n)}e^{{\rm i}\beta_1^{(n)}\hat{y}}|^2\\
 &   +|\alpha_n\beta_2^{(n)}\bar{\tilde{A}}_2^{(n)}e^{-{\rm
i}\beta_2^{(n)}\hat{y}}|^2+|\alpha_n\beta_2^{(n)}\bar{\tilde{B}}_2^{(n)}e^{{\rm
i}\beta_2^{(n)}\hat{y}}|^2
     +|\alpha_n\beta_1^{(n)}\bar{\tilde{A}}_1^{(n)}e^{-{\rm
i}\beta_1^{(n)}\hat{y}}|^2\\
&    +|\alpha_n\beta_1^{(n)}\bar{\tilde{B}}_1^{(n)}e^{{\rm
i}\beta_1^{(n)}\hat{y}}|^2
   +|\alpha_n^2\bar{\tilde{A}}_2^{(n)}e^{-{\rm i}\beta_2^{(n)}\hat{y}}|^2
    +|\alpha_n^2\bar{\tilde{B}}_2^{(n)}e^{{\rm i}\beta_2^{(n)}\hat{y}}|^2+
     |\alpha_n\beta_1^{(n)}\bar{\tilde{A}}_1^{(n)}e^{-{\rm
i}\beta_1^{(n)}\hat{y}}|^2\\
&+|\alpha_n\beta_1^{(n)}\bar{\tilde{B}}_1^{(n)}e^{{ \rm
i}\beta_1^{(n)}\hat{y}}|^2
    +|(\beta_2^{(n)})^2\bar{\tilde{A}}_2^{(n)}e^{-{\rm
i}\beta_1^{(n)}\hat{y}}|^2
    +|(\beta_2^{(n)})^2\bar{\tilde{B}}_2^{(n)}e^{{\rm
i}\beta_1^{(n)}\hat{y}}|^2
   +|(\beta_1^{(n)})^2\bar{\tilde{A}}_1^{(n)}e^{-{\rm
i}\beta_1^{(n)}\hat{y}}|^2\\
&    +|(\beta_1^{(n)})^2\bar{\tilde{B}}_1^{(n)}e^{{\rm
i}\beta_1^{(n)}\hat{y}}|^2
   +|\alpha_n\beta_2^{(n)}\bar{\tilde{A}}_2^{(n)}e^{-{\rm
i}\beta_1^{(n)}\hat{y}}|^2
    +|\alpha_n\beta_2^{(n)}\bar{\tilde{B}}_2^{(n)}
    e^{{\rm i}\beta_1^{(n)}\hat{y}}|^2\Big].
  \end{align*}
We may pick some appropriate PML parameters $\sigma$ and $\delta$ such that
$|\chi^{(n)}-\hat{\chi}^{(n)}|\leq\kappa_1^2/2$ and $|\hat{\chi}^{(n)}|\geq
\kappa_1^2/2$. It follows from the definition of $\bar{\tilde{A}}_1^{(n)}$ that
\begin{align}
|\alpha_n^2\bar{\tilde{A}}_1^{(n)}e^{-{\rm i}\beta_1^{(n)}\hat{y}}|
  \leq&\frac{|\alpha_n|}{\kappa_1^8}\Big\{\kappa_2^4
  |\alpha_n|^{5}\bigl|\varepsilon_1^{(n)}
  e^{-{\rm i}\beta_1^{(n)}\hat{y}}\bigr|^2
  +4|\alpha_n|^{5}\bigl|(\varepsilon_1^{(n)}\delta_2^{(n)}
  +2(\delta_1^{(n)}+\delta_1^{(n)}\delta_2^{(n)}))\notag\\
&\times  \beta_1^{(n)}\beta_2^{(n)}  e^{-{\rm
i}\beta_1^{(n)}\hat{y}}\bigr|^2\Big\}|v_1^{(n)}(b)|^2
  +\frac{|\alpha_n|}{\kappa_1^8}\Big\{
  \kappa_2^4|\alpha_n|^{5}\bigl|\varepsilon_1^{(n)}
  e^{-{\rm i}\beta_1^{(n)}\hat{y}}\bigr|^2\notag\\
&+4|\alpha_n|^{7}\bigl|(\varepsilon_1^{(n)}\delta_2^{(n)}+2(\delta_1^{(n)}
  +\delta_1^{(n)}\delta_2^{(n)}))
\beta_2^{(n)}e^{-{\rm i}\beta_1^{(n)}\hat{y}}\bigl|^2\Big\}|v_2^{(n)}
(b)|^2.\label{aA}
\end{align}
Since the estimates are similar for the coefficients in front of $v_1^{(n)}(b)$
and $v_2^{(n)}(b)$ in \eqref{aA}, we just present the estimates for the
coefficients in front of $v_1^{(n)}(b)$.

Again, it is necessary to consider three cases:

\begin{enumerate}

\item[(i)] If $n\in U_1$, we have $\beta_1^{(n)}=\Delta_1^{(n)}$,
 $\beta_2^{(n)}=\Delta_2^{(n)}$, $\Delta_1^{(n)}<\Delta_2^{(n)}$,
$|\alpha_n|\leq\kappa_1$, $|\beta_1^{n}|\leq\kappa_1$,
$|\beta_2^{n}|\leq\kappa_2$, and
\[
|\alpha_n|^{5/2}\bigl|\varepsilon_1^{(n)}e^{-{\rm i}\beta_1^{(n)}\hat{y}}\bigr|
  \leq\frac{2\kappa_1^{5/2}e^{\Delta_1^{(n)}({\rm Im}\hat{y}-{\rm Im}\zeta)}}
{e^{\Delta_1^{(n)}{\rm
Im}\zeta}-1}\leq\frac{2\kappa_1^{5/2}}{e^{\Delta_1^{-}{\rm Im}\zeta}-1},
\]
\begin{align*}
|\alpha_n|^{5/2}|\beta_1^{(n)}\beta_2^{(n)}|\bigl|\varepsilon_1^{(n)}\delta_2^{
(n)}e^{-{\rm i}\beta_1^{(n)}\hat{y}}\bigr|
  \leq&\frac{2\kappa_1^{7/2}\kappa_2e^{-\Delta_1^{(n)}{\rm Im}\zeta}}
  {e^{\Delta_1^{(n)}{\rm Im}\zeta}-1}
  \frac{2}{e^{\Delta_2^{(n)}{\rm Im}\zeta}-1}e^{\Delta_1^{(n)}{\rm Im}\hat{y}}\\
  \leq&\frac{4\kappa_1^{7/2}\kappa_2}
  {(e^{\Delta_1^{-}{\rm Im}\zeta}-1)(e^{\Delta_2^{-}{\rm Im}\zeta}-1)},
\end{align*}
\begin{align*}
  |\alpha_n|^{5/2}|\beta_1^{(n)}\beta_2^{(n)}|\bigl|\delta_1^{(n)}
  e^{-{\rm i}\beta_1^{(n)}\hat{y}}\bigr|
  \leq&\frac{2\kappa_1^{7/2}\kappa_2
  (e^{-\Delta_2^{(n)}{\rm Im}\zeta}+e^{-\Delta_1^{(n)}{\rm Im}\zeta})}
  {e^{\Delta_1^{(n)}{\rm Im}\zeta}-1}e^{\Delta_1^{(n)}{\rm Im}\hat{y}}\\
  \leq&\frac{4\kappa_1^{7/2}\kappa_2}
  {e^{\Delta_1^{-}{\rm Im}\zeta}-1},
\end{align*}
\begin{align*}
|\alpha_n|^{5/2}|\beta_1^{(n)}\beta_2^{(n)}|\bigl|\delta_1^{(n)}\delta_2^{(n)}e^
{-{\rm i}\beta_1^{(n)}\hat{y}}\bigr|\leq&|\alpha_n|^{5/2}|\beta_1^{(n)}\beta_2^{
(n)} \delta_2^{(n)}|\bigl|\delta_1^{(n)}e^{-{\rm i}\beta_1^{(n)}\hat{y}}\bigr|\\
  \leq&\frac{8\kappa_1^{7/2}\kappa_2}{(e^{\Delta_1^{-}{\rm
Im}\zeta}-1)(e^{\Delta_2^{-}{\rm Im}\zeta}-1)}.
\end{align*}

\item[(ii)] If $n\in U_2\backslash U_1$, we have $\beta_1^{(n)}={\rm
i}\Delta_1^{(n)}$, $\beta_2^{(n)}=\Delta_2^{(n)}$, $|\alpha_n|\leq\kappa_2$,
$\Delta_j^{(n)}\leq\kappa_2$,
and
\[
|\alpha_n|^{5/2}\bigl|\varepsilon_1^{(n)}e^{-{\rm i}\beta_1^{(n)}\hat{y}}\bigr|
  \leq\frac{2\kappa_2^{5/2}e^{\Delta_1^{(n)}({\rm Re}\hat{y}-{\rm Re}\zeta)}}
  {e^{\Delta_1^{(n)}{\rm Re}\zeta}-1}  \leq\frac{2\kappa_2^{5/2}}
  {e^{\Delta_1^{+}{\rm Re}\zeta}-1},
\]
 \begin{align*}
|\alpha_n|^{5/2}|\beta_1^{(n)}\beta_2^{(n)}|\bigl|\varepsilon_1^{(n)}\delta_2^{
(n)}e^{-{\rm i}\beta_1^{(n)}\hat{y}}\bigr|
  \leq&\frac{2\kappa_2^{9/2}e^{-\Delta_1^{(n)}{\rm Re}\zeta}}
  {e^{\Delta_1^{(n)}{\rm Re}\zeta}-1}
  \frac{2}{e^{\Delta_2^{(n)}{\rm Im}\zeta}-1}e^{\Delta_1^{(n)}{\rm Re}\hat{y}}\\
  \leq&\frac{4\kappa_2^{9/2}}
  {(e^{\Delta_1^{+}{\rm Re}\zeta}-1)(e^{\Delta_2^{-}{\rm Im}\zeta}-1)},
\end{align*}
\begin{align*}
  |\alpha_n|^{5/2}|\beta_1^{(n)}\beta_2^{(n)}|\bigl|\delta_1^{(n)}
  e^{-{\rm i}\beta_1^{(n)}\hat{y}}\bigr|
  \leq&\frac{2\kappa_2^{9/2}}
  {e^{\Delta_1^{(n)}{\rm Re}\zeta}-1}e^{\Delta_1^{(n)}{\rm Re}\hat{y}}\\
  \leq&\frac{2\kappa_2^{9/2}e^{\Delta_1^{(n)}{\rm Re}\zeta}}
  {e^{\Delta_1^{(n)}{\rm Re}\zeta}-1}
  \leq\frac{2\kappa_2^{9/2}e^{\Delta_1^{+}{\rm Re}\zeta}}
  {e^{\Delta_1^{+}{\rm Re}\zeta}-1},
\end{align*}
\begin{align*}
  |\alpha_n|^{5/2}|\beta_1^{(n)}\beta_2^{(n)}|\bigl|\delta_1^{(n)}\delta_2^{(n)}
  e^{-{\rm i}\beta_1^{(n)}\hat{y}}\bigr|
\leq&|\alpha_n|^{5/2}|\beta_1^{(n)}\beta_2^{(n)}\delta_2^{(n)}|\bigl|\delta_1^{
(n)}e^{-{\rm i}\beta_1^{(n)}\hat{y}}\bigr|\\
  \leq&\frac{4\kappa_2^{9/2}e^{\Delta_1^{+}{\rm Re}\zeta}}
  {(e^{\Delta_1^{+}{\rm Re}\zeta}-1)(e^{\Delta_2^{-}{\rm Im}\zeta}-1)}.
\end{align*}

\item[(iii)] If $n\notin U_2$, we have $\beta_1^{(n)}={\rm i}\Delta_1^{(n)}$,
$\beta_2^{(n)}={\rm i}\Delta_2^{(n)}$,
$\Delta_2^{(n)}<\Delta_1^{(n)}\leq|\alpha_n|$,
and
\[
  |\alpha_n|^{5/2}\bigl|\varepsilon_1^{(n)}e^{-{\rm
i}\beta_1^{(n)}\hat{y}}\bigr|
  \leq\frac{2|\alpha_n|^{5/2}e^{\Delta_1^{(n)}({\rm Re}\hat{y}-{\rm Re}\zeta)}}
  {e^{\Delta_1^{(n)}{\rm Re}\zeta}-1}
  \leq\frac{2|\alpha_n|^{5/2}}{e^{\frac{1}{2}\Delta_1^{(n)}}}\frac{1}
  {e^{\frac{1}{2}\Delta_1^{+}{\rm Re}\zeta}-1}
\]
\begin{align*}
  \leq&\frac{2(\kappa_1^2+25/4)^{5/4}}{e^{\frac{1}{2}\Delta_1^{+}{\rm
Re}\zeta}-1},\\
|\alpha_n|^{5/2}|\beta_1^{(n)}\beta_2^{(n)}|\bigl|\varepsilon_1^{(n)}\delta_2^{
(n)} e^{-{\rm i}\beta_1^{(n)}\hat{y}}\bigr|
  \leq&|\alpha_n|^{9/2}\bigl|\varepsilon_1^{(n)}
  e^{-{\rm i}\beta_1^{(n)}\hat{y}}\bigr||\delta_2^{(n)}|\\
  \leq&\frac{4(\kappa_1^2+81/4)^{9/4}}{(e^{\frac{1}{2}\Delta_1^{+}{\rm
Re}\zeta}-1)(e^{\Delta_2^{+}{\rm Re}\zeta}-1)},
\end{align*}
\begin{align*}
  |\alpha_n|^{5/2}|\beta_1^{(n)}\beta_2^{(n)}|\bigl|\delta_1^{(n)}
  e^{-{\rm i}\beta_1^{(n)}\hat{y}}\bigr|
  \leq&\frac{|\alpha_n|^{9/2}(e^{-\Delta_2^{(n)}{\rm Re}\zeta}
  +e^{-\Delta_1^{(n)}{\rm Re}\zeta})}
  {e^{\Delta_1^{(n)}{\rm Re}\zeta}-1}e^{\Delta_1^{(n)}{\rm Re}\hat{y}}\\
  \leq&\frac{2|\alpha_n|^{9/2}}
  {e^{\Delta_1^{(n)}{\rm Re}\zeta}-1}
  \leq\frac{2(\kappa_1^2+81/4)^{9/4}}{e^{\frac{1}{2}\Delta_1^{+}{\rm
Re}\zeta}-1},
\end{align*}
\begin{align*}
  |\alpha_n|^{5/2}|\beta_1^{(n)}\beta_2^{(n)}|\bigl|\delta_1^{(n)}\delta_2^{(n)}
  e^{-{\rm i}\beta_1^{(n)}\hat{y}}\bigr|
\leq&|\alpha_n|^{5/2}|\beta_1^{(n)}\beta_2^{(n)}\delta_2^{(n)}|\bigl|\delta_1^{
(n)}e^{-{\rm i}\beta_1^{(n)}\hat{y}}\bigr|\\
  \leq&\frac{4(\kappa_1^2+81/4)^{9/4}}{(e^{\frac{1}{2}\Delta_1^{+}{\rm
Re}\zeta}-1)(e^{\Delta_2^{+}{\rm Re}\zeta}-1)}.
\end{align*}

\end{enumerate}
We have used Proposition \eqref{eg4} in the above estimates. Combining these
estimates, we may obtain
\[
|\alpha_n^2\bar{\tilde{A}}_1^{(n)}e^{-{\rm i}\beta_1^{(n)}\hat{y}}|^2
\leq C|\alpha_n|(|v_1^{(n)}|^2+|v_2^{(n)}|^2),
\]
where the positive real number $C_1$ depends on $\kappa_j,
\Delta_j^{-},\Delta_j^{+}, {\rm Re}\zeta,$ and ${\rm Im}\zeta$. Following from
a similar argument with tedious calculations yields
\begin{align*}
    \|\nabla\tilde{\boldsymbol{v}}\|_{F(\Omega^{\rm PML})}^2
    \leq C_1\Lambda\sum\limits_{n\in\mathbb{Z}}
    |\alpha_n|(|v_1^{(n)}|^2+|v_2^{(n)}|^2),
\end{align*}
where we have used the fact $|\beta_j^{(n)}|\leq C(1+|\alpha_n|)$ for $n\in
\mathbb{Z}$. Finally, we have from Lemma \ref{tr} that
\begin{align*}
    \|\nabla\tilde{\boldsymbol{v}}\|_{F(\Omega^{\rm PML})}
    \leq C_1\|\boldsymbol{v}\|_{H^{1/2}(\Gamma)^2}
    \leq \gamma_2 C_1\|\boldsymbol{v}\|_{H^1(\Omega)^2},
\end{align*}
which completes the proof.

\section{Proof of Lemma \ref{bext}}

Taking the complex conjugate of \eqref{AB-s7} and using \eqref{do}, we have
\begin{align*}
\mathscr{D}\bar{\tilde{\boldsymbol v}}(x, b+\delta)
=- \sum_{n\in\mathbb{Z}}
&\begin{bmatrix}
          -\mu\rho\alpha_n\beta_1^{(n)}\\[5pt]
(\lambda+\mu)\alpha_n^2+(\lambda+2\mu)\bar{\rho}(\bar{\beta}_1^{(n)})^2
         \end{bmatrix}
\tilde{A}_1^{(n)}e^{-{\rm i}\bigl(\alpha_n x
-\beta_1^{(n)}\zeta\bigr)}\\
+&\begin{bmatrix}
\mu\rho\alpha_n\beta_1^{(n)}\\[5pt]
(\lambda+\mu)\alpha_n^2+(\lambda+2\mu)\rho(\beta_1^{(n)})^2
               \end{bmatrix}
\tilde{B}_1^{(n)}e^{-{\rm i}\bigl(\alpha_n x
+\beta_1^{(n)}\zeta\bigr)}\\
+&\begin{bmatrix}
\mu\rho(\beta_2^{(n)})^2\\[5pt]
-(\lambda+\mu)\alpha_n\beta_2^{(n)}
+(\lambda+2\mu)\rho\alpha_n\beta_2^{(n)}
\end{bmatrix}
\tilde{A}_2^{(n)}e^{-{\rm i}\bigl(\alpha_n x
-\beta_2^{(n)}\zeta\bigr)}\\
+&\begin{bmatrix}
\mu\rho\alpha_n(\beta_2^{(n)})^2\\[5pt]
(\lambda+\mu)\alpha_n\beta_2^{(n)}
-(\lambda+2\mu)\rho\alpha_n\beta_2^{(n)}
\end{bmatrix}
\tilde{B}_2^{(n)}e^{-{\rm i}\bigl(\alpha_n x
+\beta_2^{(n)}\zeta\bigr)}.
\end{align*}
A straightforward calculation yields that
\begin{align*}
& \|\mathscr{D}\tilde{\boldsymbol v}(x, b+\delta)\|_{L^2(\Gamma^{\rm
PML})^2}^2=\|\mathscr{D}\bar{\tilde{\boldsymbol v}}(x,
b+\delta)\|_{L^2(\Gamma^{\rm PML})^2}^2
  \leq 2\Lambda \sum\limits_{n\in\mathbb{Z}}\Big(
  \bigl|\mu\rho\alpha_n\beta_1^{(n)}\tilde{A}_1^{(n)}
  e^{{\rm i}\beta_1^{(n)}\zeta}\bigr|^2\\
&  +|\mu\rho\alpha_n\beta_1^{(n)}\tilde{B}_1^{(n)}
  e^{-{\rm i}\beta_1^{(n)}\zeta}|^2
  +|\mu\rho(\beta_2^{(n)})^2\tilde{A}_2^{(n)}
  e^{{\rm i}\beta_2^{(n)}\zeta}|^2
  +|\mu\rho(\beta_2^{(n)})^2\tilde{B}_2^{(n)}
  e^{-{\rm i}\beta_2^{(n)}\zeta}|^2\\
&+\bigl|((\lambda+\mu)\alpha_n^2+(\lambda+2\mu)\rho(\beta_1^{(n)}
)^2)\tilde{A}_1^{(n)}e^{{\rm i}\beta_1^{(n)}\zeta}\bigr|^2
+\bigl|((\lambda+\mu)\alpha_n^2+(\lambda+2\mu)\rho(\beta_1^{(n)}
)^2)\tilde{B}_1^{(n)}e^{-{\rm i}\beta_1^{(n)}\zeta}\bigr|^2\\
&  +\bigl|((\lambda+\mu)\alpha_n\beta_2^{(n)}
-(\lambda+2\mu)\rho\alpha_n\beta_2^{(n)})
  \tilde{A}_2^{(n)} e^{{\rm i}\beta_1^{(n)}\zeta}\bigr|^2
+\bigl|((\lambda+\mu)\alpha_n\beta_2^{(n)}-(\lambda+2\mu)\rho
\alpha_n\beta_2^{(n)})\tilde{B}_2^{(n)}
  e^{-{\rm i}\beta_1^{(n)}\zeta}\bigr|^2\Big).
\end{align*}
Using the similar technique in the proof of Lemma \ref{extend est} and omitting
the details, we may show that there exists a positive constant $C_2$ such that
\[
\|\mathscr{D}\tilde{\boldsymbol v}(x, b+\delta)\|_{L^2(\Gamma^{\rm PML})^2}^2\\
  \leq C_2\sum\limits_{n\in\mathbb{Z}}\bigl[(1+|\alpha_n|)
  (|v_1^{(n)}(b)|^2+|v_1^{(n)}(b)|^2)\bigr].
\]
Finally, it follows from Lemma \ref{tr} that
\begin{align*}
\|\mathscr{D}\tilde{\boldsymbol v}(x, b+\delta)\|_{L^2(\Gamma^{\rm PML})^2}
   \leq C_2\|\boldsymbol{v}\|_{H^{1/2}(\Gamma)^2}
    \leq \gamma_2 C_2\|\boldsymbol{v}\|_{H^1(\Omega)^2},
\end{align*}
which completes the proof.

\end{document}